\documentclass[11pt,english]{amsart}
\usepackage[T1]{fontenc}
\usepackage[latin1]{inputenc}
\usepackage{amstext}
\usepackage{amsthm}
\usepackage{amssymb}

\makeatletter

\providecommand{\tabularnewline}{\\}

\numberwithin{equation}{section}
\numberwithin{figure}{section}
\theoremstyle{plain}
\newtheorem{thm}{\protect\theoremname}
\theoremstyle{plain}
\newtheorem{prop}[thm]{\protect\propositionname}
\theoremstyle{plain}
\newtheorem{cor}[thm]{\protect\corollaryname}
\theoremstyle{remark}
\newtheorem*{rem*}{\protect\remarkname}
\theoremstyle{definition}
\newtheorem{example}[thm]{\protect\examplename}
\theoremstyle{remark}
\newtheorem{rem}[thm]{\protect\remarkname}

\makeatother

\makeatother

\usepackage{babel}
\providecommand{\corollaryname}{Corollary}
\providecommand{\examplename}{Example}
\providecommand{\propositionname}{Proposition}
\providecommand{\remarkname}{Remark}
\providecommand{\theoremname}{Theorem}

\begin{document}

%\title[On the geometry of some K3 surfaces]{On the geometry of some K3 surfaces with finite automorphism group}

%\title[On the geometry of some K3 surfaces]{On the geometry of  K3 surfaces with rational polyhedral effective cone: the compact case}

\title[On the geometry of some K3 surfaces]{On the geometry of  K3 surfaces with finite automorphism group and no elliptic fibrations}

%\title[On the geometry of some K3 surfaces]{On the geometry of  K3 surfaces with finite automorphism group: the compact case}

\addtolength{\textwidth}{5mm}
\addtolength{\hoffset}{-5mm} 
\addtolength{\textheight}{12mm}
\addtolength{\voffset}{-5mm} 

%\subjclass{Primary: 14J29} ; Secondary: 14G10, 14G15}

\global\long\def\CC{\mathbb{C}}%
 
\global\long\def\BB{\mathbb{B}}%
 
\global\long\def\PP{\mathbb{P}}%
 
\global\long\def\QQ{\mathbb{Q}}%
 
\global\long\def\RR{\mathbb{R}}%
 
\global\long\def\FF{\mathbb{F}}%

\global\long\def\DD{\mathbb{D}}%
 
\global\long\def\NN{\mathbb{N}}%
\global\long\def\ZZ{\mathbb{Z}}%
 
\global\long\def\HH{\mathbb{H}}%
 
\global\long\def\Gal{{\rm Gal}}%
\global\long\def\OO{\mathcal{O}}%
\global\long\def\pP{\mathfrak{p}}%

\global\long\def\pPP{\mathfrak{P}}%
 
\global\long\def\qQ{\mathfrak{q}}%
 
\global\long\def\mm{\mathcal{M}}%
 
\global\long\def\aaa{\mathfrak{a}}%
 
\global\long\def\a{\alpha}%
 
\global\long\def\b{\beta}%
 
\global\long\def\d{\delta}%
 
\global\long\def\D{\Delta}%
 
\global\long\def\L{\Lambda}%
 
\global\long\def\g{\gamma}%

\global\long\def\G{\Gamma}%
 
\global\long\def\d{\delta}%
 
\global\long\def\D{\Delta}%
 
\global\long\def\e{\varepsilon}%
 
\global\long\def\k{\kappa}%
 
\global\long\def\l{\lambda}%
 
\global\long\def\m{\mu}%
 
\global\long\def\o{\omega}%
 
\global\long\def\p{\pi}%
 
\global\long\def\P{\Pi}%
 
\global\long\def\s{\sigma}%

\global\long\def\S{\Sigma}%
 
\global\long\def\t{\theta}%
 
\global\long\def\T{\Theta}%
 
\global\long\def\f{\varphi}%
 
\global\long\def\deg{{\rm deg}}%
 
\global\long\def\det{{\rm det}}%

\global\long\def\Dem{Proof: }%
 
\global\long\def\ker{{\rm Ker\,}}%
 
\global\long\def\im{{\rm Im\,}}%
 
\global\long\def\rk{{\rm rk\,}}%
 
\global\long\def\car{{\rm car}}%
\global\long\def\fix{{\rm Fix( }}%

\global\long\def\card{{\rm Card\  }}%
 
\global\long\def\codim{{\rm codim\,}}%
 
\global\long\def\coker{{\rm Coker\,}}%
 
\global\long\def\mod{{\rm mod }}%

\global\long\def\pgcd{{\rm pgcd}}%
 
\global\long\def\ppcm{{\rm ppcm}}%
 
\global\long\def\la{\langle}%
 
\global\long\def\ra{\rangle}%

\global\long\def\Alb{{\rm Alb(}}%
 
\global\long\def\Jac{{\rm Jac(}}%
 
\global\long\def\Disc{{\rm Disc(}}%
 
\global\long\def\Tr{{\rm Tr(}}%
 
\global\long\def\NS{{\rm NS(}}%
 
\global\long\def\Pic{{\rm Pic(}}%
 
\global\long\def\Pr{{\rm Pr}}%

\global\long\def\Km{{\rm Km(}}%
\global\long\def\rk{{\rm rk(}}%
\global\long\def\Hom{{\rm Hom(}}%
 
\global\long\def\End{{\rm End(}}%
 
\global\long\def\aut{{\rm Aut}}%
 
\global\long\def\SSm{{\rm S}}%
\global\long\def\psl{{\rm PSL}}%
\global\long\def\arccosh{{\rm Arccosh}}%
\global\long\def\cu{{\rm (-2)}}%

\author{Xavier Roulleau}
\begin{abstract}
Nikulin \cite{Nikulin1}, \cite{Nikulin2} and Vinberg \cite{Vinberg}
proved that there are only a finite number of lattices of rank $\geq3$
that are the Néron-Severi lattice of projective K3 surfaces with a
finite automorphism group. The aim of this paper is to provide a more
geometric description of such K3 surfaces $X$, when these surfaces
have moreover no elliptic fibrations. In that case we show that such
K3 surface is either a quartic with special hyperplane sections or
a double cover of the plane branched over a smooth sextic curve which
has special tangencies properties with some lines, conics or cuspidal
cubic curves. We then study the converse i.e. if the geometric description
we obtained characterizes these surfaces. In $4$ cases the description
is sufficient, in each of the $4$ other cases there is exactly another
one possibility which we study.  We obtain that at least $5$ moduli
spaces of K3 surfaces (among the $8$ we study) are unirational.
\end{abstract}


\maketitle

\section{Introduction}

The classification of projective complex K3 surfaces $X$ that have
a finite automorphism group has been established by Pyatetskii-Sapiro
and Shafarevich \cite{PS} for Picard ranks $\rho=\rk\NS X))$ $1$
and $2$, by Vinberg \cite{Vinberg} for rank $4$ and Nikulin \cite{Nikulin1}, \cite{Nikulin2}
for the other ranks. This classification is done according to their
Néron-Severi lattice $\NS X)$, which, according to the Torelli Theorem
(see \cite{PS}) for K3 surfaces, strongly characterizes the surfaces.
However in general for rank $\rho>1$, we do not have a geometric
construction of these surfaces. By geometric construction, we mean
surfaces lying in a projective space or understood as cover of a known
surface or, if there exist a fibration with a section, an explicit
Weierstrass model. \\
By results of Nikulin, a K3 surface $X$ with $\rho>2$ has a finite
number of automorphism if and only if the number of smooth rational
curves (called $(-2)$-curves) is finite and non-zero. According to
\cite{AHL} and independently \cite{McK}, these conditions are equivalent
to requiring the surface $X$ to be a Mori dream space i.e. that its
Cox ring is finitely generated. 

In this paper, we construct geometrically the K3 surfaces $X$ which
have a finite non-zero number of $(-2)$-curves and no elliptic fibrations
(the case with fibrations is treated in \cite{Roulleau}). From the
results of Nikulin \cite{Nikulin2} and Vinberg \cite{Vinberg}, their
Picard number $\rho$ is $3$ or $4$ and there are $8$ cases classified
according to their Néron-Severi lattices, which are denoted $S_{1},\dots,S_{6}$
for rank $\rho=3$ and $L(24),L(27)$ for rank $\rho=4$. 

A conjecture which remained open for a quite long time is about the
presence of infinitely many rational curves on a K3 surface. We observe
that some of the last cases that remained open (and for which now
the conjecture was proved, see e.g. \cite{CGL} for a history and
references) are the cases of K3 surfaces with Picard number 4, no
elliptic fibrations and a finite automorphisms group. These K3 surfaces
are exactly the surfaces we consider in Sections \ref{subsec:Lattice-L24}
and \ref{subsec:Lattice-L27}, and probably the difficulties in establishing
the conjecture in these cases is related to their peculiar geometry.
Also from the point of view of this paper, these are special K3 surfaces
since both of them have a twin (see below for the definition) with
different properties with respect to the presence of $(-2)$-curves.

In \cite{Kondo}, Kondo studies the automorphism group of generic
K3 surfaces with a finite number of $(-2)$-curves; for Picard rank
$\leq8$ this group is either trivial or $\ZZ/2\ZZ$. We obtain: 
\begin{thm}
\label{thm:Main1} Let $X$ be a $K3$ surface of Picard number $>2$
with finite automorphism group and no elliptic fibrations. 
\begin{itemize}
\item The automorphism group of $X$ is trivial if and only $X$ can be
embedded as a smooth quartic in $\PP^{4}$. Then that embedding is
unique and the quartic $X$ is one the following two types:\\
i) Type $S_{2}$ ($\rho=3$). There are three quadrics in $\PP^{3}$
such that each intersects $X\hookrightarrow\PP^{3}$ in the union
of two smooth rational quartic curves. These $6$ rational curves
are the only $(-2)$-curves on $X$.\\
ii) Type $S_{3}$ ($\rho=3$). There exist two hyperplanes sections
such that each hyperplane section is a union of two smooth conics.
These $4$ conics are the only $(-2)$-curves on $X$. 
\item The group $\aut(X)$ is isomorphic to $\ZZ/2\ZZ$ if and only if $X$
is a double cover of the plane branched over a smooth sextic curve
$C_{6}$. In that case the possibilities are:\\
1) Type $S_{1}$ ($\rho=3$). The $6$ $(-2)$-curves on $X$ are
pull-back of $3$ conics that are $6$-tangent to the sextic $C_{6}$.\\
2) Type $S_{4}$ ($\rho=3$). The $4$ $(-2)$-curves on $X$ are
pull-back of a tritangent line and a $6$-tangent conic to $C_{6}$.\\
3) Type $S_{5}$ ($\rho=3$). The $4$ $(-2)$-curves on $X$ are
pull-back of two tritangent lines to $C_{6}$.\\
4) Type $S_{6}$ ($\rho=3$). The $6$ $(-2)$-curves on $X$ are
pull-back of one $6$-tangent conic and two cuspidal cubics that cut
$C_{6}$ tangentially and at their cusps.\\
5) Type $L(24)$ ($\rho=4$). The $6$ $(-2)$-curves on $X$ are
pull-back of three tritangent lines to $C_{6}$.\\
6) Type $L(27)$ ($\rho=4$). The $8$ $(-2)$-curves on $X$ are
pull-back of one tritangent line and three $6$-tangent conics to
$C_{6}$.
\end{itemize}
\end{thm}

Here we say that a line (respectively a conic) is tritangent (respectively
$6$-tangent) to a sextic if the intersection multiplicity is even
at every intersections points of the line (resp. the conic) and the
sextic. 

We implemented some algorithm in order to obtain classes of big and
nef divisors on K3 surfaces. By exploiting the fact that the fundamental
domain for the Weyl group is bounded, we can obtain all big and nef
(or even ample) classes of square less than or equal to a given bound.
In particular we can compute the first terms of the generating series
\[
\Xi_{X}(T)=\sum_{D\in\NS X)\,\text{big, nef}}T^{D^{2}}.
\]
For example let $X$ be a K3 surface of type $S_{2}$ in Theorem \ref{thm:Main1}.
Then
\begin{thm}
\label{thm:The-generating-series}The generating series $\Xi_{X}$
begins with 
\[
\begin{array}{c}
T^{4}+6T^{10}+6T^{12}+T^{16}+6T^{18}+12T^{22}+6T^{28}+18T^{30}+6T^{34}+7T^{36}\\
+6T^{40}+12T^{46}+6T^{48}+18T^{58}+T^{64}+12T^{66}+12T^{70}+6T^{72}+6T^{76}\\
+18T^{82}+12T^{88}+18T^{90}+12T^{94}+7T^{100}+24T^{102}+18T^{106}+6T^{108}...
\end{array}
\]
\end{thm}

Theorem \ref{thm:Main1} gives possible realizations of K3 surfaces
$X$ with a finite number of automorphisms and no elliptic fibrations
as double cover of the plane branched over a particular smooth sextic
curve, or as particular quartic surfaces in $\PP^{3}$. In the last
section of this paper, we study the converse, i.e. if the geometric
description of the K3 surfaces given in Theorem \ref{thm:Main1} are
enough to characterize them as surfaces in the list $S_{1},\dots,L(27)$.
It turns out that in four cases among the eight, there are two possibilities.
Let us give an example:
\begin{prop}
Let $Y$ be the K3 surface which is the double cover of the plane
$\pi:Y\to\PP^{2}$ branched over a generic sextic curve that possess
$3$ tritangent lines. The pull-back of each $3$ lines splits as
the union of two $(-2)$-curves. Up to permutation, the intersection
matrix of the six $(-2)$-curves over the lines is one of the following
matrices
\[
\left(\begin{array}{cccccc}
-2 & 3 & 0 & 1 & 0 & 1\\
3 & -2 & 1 & 0 & 1 & 0\\
0 & 1 & -2 & 3 & 0 & 1\\
1 & 0 & 3 & -2 & 1 & 0\\
0 & 1 & 0 & 1 & -2 & 3\\
1 & 0 & 1 & 0 & 3 & -2
\end{array}\right),\,\quad\left(\begin{array}{cccccc}
-2 & 3 & 0 & 1 & 0 & 1\\
3 & -2 & 1 & 0 & 1 & 0\\
0 & 1 & -2 & 3 & 1 & 0\\
1 & 0 & 3 & -2 & 0 & 1\\
0 & 1 & 1 & 0 & -2 & 3\\
1 & 0 & 0 & 1 & 3 & -2
\end{array}\right).
\]
In the first case, the surface $Y$ has type $L(24)$. In the second
case, $Y$ contains also only six $(-2)$-curves, but there exists
a fibration. Both cases exist.
\end{prop}

Let $X$ be a K3 surface of type in the list $S_{1},\dots,L(27)$.
Let us say that a K3 surface $X'$ is a twin if it has the same geometric
description as in Theorem \ref{thm:Main1}, the same Picard number,
but a different Néron-Severi lattice. We obtain:
\begin{thm}
The K3 surfaces of types $S_{3},S_{4},L(24),L(27)$ have twins and
the Néron-Severi lattices of each twin is unique. Surfaces of types
$S_{1},S_{2},S_{5},S_{6}$ are characterized by their description
in Theorem \ref{thm:Main1}.
\end{thm}

The surfaces that are twin with surfaces of type $S_{3},S_{4},L(24),L(27)$
have some elliptic fibrations, three of these also have a finite automorphism
group; these surfaces are discussed with more details in \cite{Roulleau}. 

In the last section, we also give some explicit examples of K3 surfaces
of various type, using the method of van Luijk \cite{vL} refined
by Elsenhans and Jahnel \cite{EJ} in order to check their Picard
number. We moreover obtain that
\begin{thm}
The moduli spaces of K3 surfaces of type $S_{1},S_{2},S_{3},S_{5}$
and $S_{6}$ are unirational. 
\end{thm}

This is done by constructing the families of these examples. In the
Appendix, we give various datas of the K3 surfaces we are considering,
(such that their Néron-Severi lattices). The reader can find the Magma
programs used for the computations into the ancillary file in the
arXiv version of this paper. 

\textbf{Acknowledgements} The author is grateful to the referees for
many suggestions, comments and ideas improving the paper. The author
is also thankful to Edgar Costa for his computation of the zeta functions
used in Examples \ref{exa:COSTA_S2}, \ref{exa:Costa} and to Carlos
Rito for discussions on the construction of some sextic curves. 

\section{Preliminaries}

\subsection{On the automorphisms of K3 surfaces with low Picard number}

Let $X$ be a K3 surface with finite automorphism group and Picard
number $\leq8$. The following result is \cite[Main Theorem, Lemma 2.3]{Kondo}:
\begin{prop}
\label{prop:KONDO-symplectic-automorphism} The group of symplectic
automorphisms is trivial. If moreover $X$ is general then the group
$\aut(X)$ is the trivial group or it is generated by a non symplectic
involution. 
\end{prop}

In case of Picard rank $3$, one can remove the hypothesis that $X$
is general. Indeed a non-symplectic automorphism should have order
$m$ such that $\varphi(m)$ divides $22-3=19$ (where $\varphi$
is the totient Euler function), and that implies $m=2$. 

\subsection{Linear series on K3 surfaces}

 To be self-contained, let us recall the following results of Saint
Donat \cite{SaintDonat}: 
\begin{thm}
\label{thm:SaintDonat}Let $X$ be a K3 surface that does not possess
an elliptic pencil and let $D$ be a big and nef divisor on $X$.\\
a) One has $h^{0}(D)=2+\frac{1}{2}D^{2}$ and the linear system $|D|$
is base point free,\\
b) the morphism $\varphi_{|D|}$ associated to $|D|$ is either $2$
to $1$ onto its image (hyperelliptic case) or it maps $X$ birationally
onto its image: in both the cases it contracts the $(-2)$-curves
$\G$ such that $D\G=0$. \\
c) if $|D|$ is hyperelliptic, then $\varphi_{|D|}$ is a double cover
of $\PP^{2}$ or of the Veronese surface in $\PP^{5}$. \\
d) the linear system $|D|$ is hyperelliptic with $D^{2}\geq4$ if
and only if there is a nef divisor $B$ with $B^{2}=2$ and $D=2B$.\\
e) if $|D|$ is hyperelliptic and $D^{2}=2$, then $\varphi_{|D|}$
is a double cover of $\PP^{2}$.
\end{thm}

A K3 surface with no elliptic pencil cannot have nef classes with
self-intersection $0$. 

\subsection{Shimada algorithm to compute classes with fixed square and degree\label{subsec:An-algorithm-to}}

Let $X$ be a smooth surface. Let $d,k$ be positive integers and
let $H$ be an ample class. The following is an algorithm for computing
all classes $D$ with $D^{2}=d$ and degree $HD=k>0$. First one projects
onto the orthogonal $H^{\perp}$ of $H$ by the linear map 
\[
\pi:D\in\NS X)\to(H^{2})D-(DH)H\in H^{\perp}.
\]
The class $D$ has square $D^{2}=d$ and degree $k=HD$ if and only
if 
\[
\pi(D)^{2}=(H^{2})^{2}D^{2}-(H^{2})(DH)^{2}=(H^{2})^{2}d-k^{2}H^{2}\text{ and }DH>0.
\]
Since $H$ is ample, $H^{\perp}$ is a negative definite lattice,
thus when $d>0$ and $k>0$ are fixed, there are only a finite number
of solutions $s\in H^{\perp}$ with $s^{2}=(H^{2})^{2}d-k^{2}H^{2}$,
which can be found by a computer calculation. Then we take the reverse
path to the Néron-Severi lattice by the application $s\to\frac{1}{H^{2}}(s+kH)$,
discarding solutions $D$ that are not in the integral lattice $\NS X)$
or such that $DH<0$. 

When $X$ is a K3 surface with a finite number of $(-2)$-curves it
is then easy to check if a divisor class $D$ in $\NS X)$ is nef
or even ample. If moreover the fundamental domain of the Weyl group
is bounded, we explain in the next section how we can compute all
classes of a given square, without restriction on their degree. 

In fact one can understand the above algorithm as a easy proof that
on a surface, the Hilbert Scheme of curves of given degree and square
with respect to a polarization has only a finite number of irreducible
components. The idea of it came from an algorithm of Lairez and Sertoz
in \cite{Lairez} (following an idea of Deygtarev) for computing $(-2)$-curves,
but soon after completion of this paper appeared \cite{DS} where
it turns out that it was already described and used by Shimada in
\cite{Shimada}. We used Magma software \cite{magma} for the computations. 

We use the above algorithm to compute the $(-2)$-classes, and Vinberg's
sieve (see \cite{VinbergBombay}) to sort the $(-2)$-curves among
the $(-2)$-classes. 

\subsection{Néron-Severi lattice and hyperbolic geometry\label{subsec:N=0000E9ron-Severi-hyperbolic}}

Let $X$ be a surface of Picard number $\rho\geq2$. Let $\PP^{\rho-1}$
be the projectivisation of the real space $\NS X)\otimes\RR$. The
image $\mathcal{L}_{X}$ in $\PP^{\rho-1}$ of the positive cone $\{c\in\NS X)_{\RR}\,|\,c^{2}>0\}$
is a ball of dimension $\rho-1$. We will often not make a notational
distinction between classes in $\NS X)\otimes\RR$ and in $\PP^{\rho-1}$,
also we will often confuse a primitive effective divisor with its
image in $\mathcal{L}_{X}$. For $D,D'\in\mathcal{L}_{X}$ let us
define 
\[
\ell_{X}(D,D')=\frac{(DD')^{2}}{D^{2}D'^{2}}.
\]
By the Hodge Index Theorem $\ell_{X}(D,D')\geq1$ with equality if
and only if $D$ and $D'$ are proportional (observe that the function
$\ell$ is homogeneous). In fact, the following function: 
\[
d_{X}(D,D')=\arccosh\left(\sqrt{\ell_{X}(D,D')}\right),\,\,D,D'\in\mathcal{L}_{X}
\]
is a distance on $\mathcal{L}_{X}$ for which the geodesic between
two points in $\mathcal{L}_{X}$ is the intersection of $\mathcal{L}_{X}$
with the line in $\PP^{\rho-1}$ trough these points (see \cite[Chapter 3, Section 3.2]{Ratcliffe}). 

Let us suppose that the fundamental domain $\mathcal{F}_{X}$ of the
Weyl group is compact. Let us fix an ample element $H$ in $\mathcal{F}_{X}$,
let $d_{max}(\mathcal{F}_{X},H)$ be the maximum of the distances
between $H$ and the points in $\mathcal{F}_{X}$ and let us define
$\ell(\mathcal{F}_{X},H)=\cosh(d_{max}(\mathcal{F}_{X},H))$. Then
for any big and nef classes $D$ in $\NS X)$, one has 
\[
(HD)^{2}\leq\ell(\mathcal{F}_{X},H)H^{2}D{}^{2}.
\]
The consequence is that for the surfaces with compact $\mathcal{F}_{X}$,
we can apply the algorithm in Section \ref{subsec:An-algorithm-to}
to find all nef and big classes $D$ in $\NS X)$ with fixed square
$D^{2}=d$: it is enough to search classes $D$ with $D^{2}=d$ up
to degree $\sqrt{\ell(\mathcal{F}_{X},H)dH^{2}}$ with respect to
the ample class $H$. Then we can compute the generating series 
\[
\Theta_{X}(T)=\sum_{D\text{ big,nef,primitive}}T^{D^{2}}
\]
of the primitive nef and big classes in $\NS X)$, up to some square.
\\
The distances $\ell(\mathcal{F}_{X},H)$ are computed by using the
extremal rays $r_{1},\dots,r_{k}$ of the nef cone (the hyperplanes
$r_{j}^{\perp}$ are the facets of the effective cone). For example,
in case of rank $3$, the rays of the nef cone are elements that are
orthogonal to two $(-2)$-curves $A_{i},A_{j}$ with $A_{i}A_{j}\in\{0,1\}$. 

Any even lattice of rank $\geq5$ represents $0$ i.e. there is a
non-trivial class $c$ with $c^{2}=0$, thus the fundamental domain
of a K3 surface with Picard number $\geq5$ cannot be bounded. Moreover
if the fundamental domain of a K3 surface $X$ is bounded, then $X$
contains only a finite number of $(-2)$-curves. Indeed if there are
an infinite number of $(-2)$-curves then the automorphism group of
$X$ is infinite by the Theorem of Sterk \cite{Sterk}. Therefore
there exist big and nef classes which have an infinite orbit with
the same square; but by the above discussion such a phenomenon is
not possible on a bounded domain. 

Finally, let us mention \cite{ShimadaLatt}, where algorithms to study
pull-back of divisors by double covers are developed and can be used
in this paper. 

\section{Rank 3 and compact cases}

In this Section we study K3 surfaces with Picard rank $3$ such that
the fundamental polygon for the Weyl group is compact, in particular
these surfaces have no elliptic fibrations. According to \cite{Nikulin2},
there are $6$ types of lattices $S_{1},\dots,S_{6}$ for the Néron-Severi
lattice of such surfaces; Nikulin described their associated fundamental
domain $\mathcal{F}_{X}$. 

\subsection{Surfaces with Néron-Severi lattice of type $S_{1}$}

Let us study K3 surfaces $X$ with a finite number of $(-2)$-curves
and such that their Néron-Severi lattice is generated by $L,A_{1},A_{2}$,
with $L$ nef and $A_{1},A_{2}$ two $(-2)$-curves, with intersection
matrix 
\[
\left(\begin{array}{ccc}
6 & 0 & 0\\
0 & -2 & 0\\
0 & 0 & -2
\end{array}\right).
\]
The surface $X$ contains exactly $6$ $(-2)$-curves $A_{1},\dots,A_{6}$.
In the base $A_{1},A_{2},A_{3}=L-2A_{1}$, the classes of the $(-2)$-curves
$A_{4},A_{5},A_{6}$ are
\[
A_{4}=A_{1}-2A_{2}+2A_{3},\;A_{5}=2A_{1}-3A_{2}+2A_{3},\;A_{6}=2A_{1}-2A_{2}+A_{3}.
\]
 The intersection matrix of the $6$ curves is
\[
\left(\begin{array}{cccccc}
-2 & 0 & 4 & 6 & 4 & 0\\
0 & -2 & 0 & 4 & 6 & 4\\
4 & 0 & -2 & 0 & 4 & 6\\
6 & 4 & 0 & -2 & 0 & 4\\
4 & 6 & 4 & 0 & -2 & 0\\
0 & 4 & 6 & 4 & 0 & -2
\end{array}\right).
\]
Using Section \ref{subsec:N=0000E9ron-Severi-hyperbolic}, we find
that there is a unique big and nef class such that 
\[
D^{2}=2,
\]
which is 
\[
D=A_{1}-A_{2}+A_{3}=L-A_{1}-A_{2}.
\]
The degree of the curves $A_{i}$ with respect to $D$ is $DA_{i}=2$
(thus $D$ is ample). Since $X$ has no elliptic fibration, the linear
system $|D|$ is free and $\varphi_{|D|}$ is a double cover of the
plane. We remark that 
\begin{equation}
\begin{array}{c}
A_{1}+A_{4}=2D\\
A_{2}+A_{5}=2D\\
A_{3}+A_{6}=2D,
\end{array}\label{eq:S1}
\end{equation}
so that $A_{1}+A_{4},\,A_{2}+A_{5}$ and $A_{3}+A_{6}$ are pull-back
to $X$ of conics $C_{1},C_{2},C_{3}$ in the plane. The branch curve
$B$ is a smooth sextic and, from equations \eqref{eq:S1}, each conic
$C_{j}$ is tangent to $B$ at $6$ points. 
\begin{prop}
\label{prop:The-surface-S1}The surface $X$ is a double cover of
$\PP^{2}$ branched over a smooth sextic curve $B$. There exists
$3$ conics that are $6$-tangent to $B$. The six $(-2)$-curves
on $X$ are the pull-back of these $3$ conics.\\
The automorphism group of $X$ is $\ZZ/2\ZZ$. 
\end{prop}

The automorphism group of $X$ is generated by the involution $\iota$
from the double-plane structure. From Equation \ref{eq:S1}, the action
of $\iota^{*}$ on $\NS X)$ is by $\iota^{*}A_{j}=A_{j+3\text{ mod }6}$.
The Néron-Severi lattice contains exactly $6$ big and nef classes
$D_{i}$ with $D_{i}^{2}=6$. Each of these classes contracts two
disjoint $(-2)$-curves $A_{i},A_{i+1}$ (the index being taken mod
$6$) and they are the extremal points of the polyhedral fundamental
domain $\mathcal{F}_{S_{1}}$ (see \cite[Figure 1]{Nikulin2}). These
classes $D_{1},\dots,D_{6}$ are: 
\[
\begin{array}{c}
D_{1}=2A_{1}+A_{3},\,\,D_{2}=A_{1}+2A_{3},\,\,D_{3}=2A_{1}-3A_{2}+4A_{3},\\
D_{4}=4A_{1}-6A_{2}+5A_{3},\,\,D_{5}=5A_{1}-6A_{2}+4A_{3},\,\,D_{6}=4A_{1}-3A_{2}+2A_{3},
\end{array}
\]
and their degree with respect to $D$ is $DD_{i}=6$ (moreover $\iota^{*}D_{j}=D_{j+3\text{ mod }6}$).
Therefore the distance between $D$ and these divisors $D_{i}$ equals
to $\arccosh(\sqrt{3})$, which is the maximal distance between $D$
and any points in $\mathcal{F}_{S_{1}}$. Using that, we find that
the generating series 
\[
\Theta_{X}(T)=\sum_{D\text{ big,nef,primitive}}T^{D^{2}}
\]
of the primitive nef and big classes begins with 
\[
\begin{array}{c}
T^{2}+6(T^{4}+T^{6}+T^{14}+T^{20}+2T^{22}+T^{28}+T^{34}+T^{38}+2T^{44}+3T^{23}+T^{50}\\
+2T^{52}+T^{60}+T^{62}+T^{68}+2T^{70}+2T^{76}+2T^{78}+T^{82}+3T^{86}+2T^{92})...
\end{array}
\]
The $6$ classes $A_{1}+A_{3},A_{1}+A_{5},A_{2}+A_{4},A_{2}+A_{6},A_{3}+A_{5},A_{4}+A_{6}$
are nef divisors of square $4$. Each of these $6$ classes is orthogonal
to a unique $(-2)$-curve. Such a class $C$ realizes $X$ as a $1$-nodal
quartic in $\PP^{3}$. Projecting from a node, one get a double-cover
from $X$ to $\PP^{2}$. Since the automorphism group of $X$ is $\ZZ/2\ZZ$,
that double cover must be the one in Proposition \ref{prop:The-surface-S1},
that phenomenon is a bit surprising, but is not a contradiction. \\
Also, equality $D=L-A_{1}-A_{2}$ means that the model of $\varphi_{|D|}$
is geometrically obtained as a projection of the model in $\PP^{4}$
given by $\varphi{}_{|L|}$ obtained by projecting from the line connecting
the two nodes $\varphi_{|D|}(A_{1})$ and $\varphi_{|D|}(A_{2})$.

\subsection{Surfaces with Néron-Severi lattice of type $S_{2}$}

Let us study K3 surfaces with a finite number of $(-2)$-curves and
such that the Néron-Severi lattice is generated by $L,A_{1},A_{2}$
with $L$ nef and $A_{1},A_{2}$ two $(-2)$-curves with intersection
matrix 
\[
\left(\begin{array}{ccc}
36 & 0 & 0\\
0 & -2 & 1\\
0 & 1 & -2
\end{array}\right).
\]
Then $X$ contains exactly $6$ $(-2)$-curves $A_{1},\dots,A_{6}$.
The $(-2)$-curve $A_{3}=L-5A_{1}-3A_{2}$ is such that $A_{1},A_{2},A_{3}$
generate $\NS X)$ and the classes of the $3$ other $(-2)$-curves
are 
\[
A_{4}=A_{1}-2A_{2}+2A_{3},\,\,A_{5}=2A_{1}-3A_{2}+2A_{3},\,\,A_{6}=2A_{1}-2A_{2}+A_{3}.
\]
The intersection matrix of the $6$ curves $A_{j}$ is 
\[
\left(\begin{array}{cccccc}
-2 & 1 & 7 & 10 & 7 & 1\\
1 & -2 & 1 & 7 & 10 & 7\\
7 & 1 & -2 & 1 & 7 & 10\\
10 & 7 & 1 & -2 & 1 & 7\\
7 & 10 & 7 & 1 & -2 & 1\\
1 & 7 & 10 & 7 & 1 & -2
\end{array}\right).
\]
 Using Section \ref{subsec:N=0000E9ron-Severi-hyperbolic}, we obtain
that the class
\[
D=L-4A_{1}-4A_{2}=A_{1}-A_{2}+A_{3}.
\]
is the unique big and nef class with square
\[
D^{2}=4.
\]
There are no big and nef divisors $D'$ with $D'^{2}<4$. For each
$(-2)$-curve $A_{j}$ on $X$, we have $A_{j}D=4$, thus $D$ is
ample. Since there is no elliptic fibration, by Theorem \ref{thm:SaintDonat}
the class $D$ is a very ample. We remark that 
\[
\begin{array}{c}
A_{1}+A_{4}=2D\\
A_{2}+A_{5}=2D\\
A_{3}+A_{6}=2D.
\end{array}
\]
Therefore
\begin{prop}
There exists a unique embedding of $X$ in $\PP^{3}$. For that embedding,
there exist $3$ quadrics $Q_{1},Q_{2},Q_{3}$ and each of them cuts
$X\subset\PP^{3}$ into the union of two $(-2)$-curves, each of degree
$4$. 
\end{prop}

Suppose that there is a non-symplectic involution acting on $X$.
Then its fixed locus has $1$ dimensional components (see \cite{AST}).
Since there is only one degree $4$ polarization, the involution must
preserve it. The fixed locus of a non symplectic involution on a smooth
quartic surface contains either a line or a smooth quartic curve.
The first case is impossible since the $(-2)$-curves on $X$ have
degree 4. By \cite[Table 1]{AST}, in the second case the Picard number
of $X$ would have rank at least $8$, it is therefore also impossible.
Thus according to Propositions \ref{prop:KONDO-symplectic-automorphism}:
\begin{prop}
The automorphism group of $X$ is trivial. 
\end{prop}

The linear maps $g,h:\NS X)\to\NS X)$ such that $g(A_{1})=A_{2}$,
$g(A_{2})=A_{3}$, $g(A_{3})=A_{4}$ and $h(A_{1})=A_{1}$, $h(A_{2})=A_{6}$,
$h(A_{3})=A_{5}$ are isometries (of respective order $6$ and $2$)
of the lattice $\NS X)$. They verify $hgh=g^{-1}$, thus they generate
the dihedral group $\DD_{6}$ of order $12$, the symmetry group of
the hexagon. That group preserves the set of classes $A_{i}$, the
ample cone and the fundamental domain for the Weyl group $W_{X}$,
therefore there is no non-trivial element of $\DD_{6}$ in the Weyl
group $W_{X}$. Since the automorphism group of $X$ is trivial, no
isometry in $\DD_{6}$ comes from an automorphism of $X$, therefore:
\begin{cor}
The index of the Weyl group $W_{X}$ in $O(\NS X))$ is $12$.
\end{cor}

Since $LA_{1}=LA_{2}=0$ and $L=5A_{1}+3A_{2}+A_{3}$ is big and nef,
the class $L$ with ($L^{2}=36$) is an extremal point of the fundamental
polygon $\mathcal{F}_{X}$, (this is also in \cite[Table 2]{Nikulin2}).
The coordinates with respect to the basis $(A_{1},A_{2},A_{3})$ of
the extremal classes $L=L_{1},L_{2},\dots,L_{6}$ of $\mathcal{F}_{X}$
are:
\[
(5,3,1),(13,-9,5),\,(17,-21,13),\,(13,-21,17),\,(5,-9,13),\,(1,3,5).
\]
One has $L_{i}A_{2-i}=L_{i}A_{3-i}=0$, where the indices are modulo
$6$. Since $DL_{i}=36$ for all $i\in\{1,...,6\}$, the maximum distance
from $D$ in $\mathcal{F}_{X}$ equals 
\[
d_{max}=\arccosh\left(\sqrt{\frac{(LD)^{2}}{L^{2}D^{2}}}\right)=\arccosh(3).
\]
Using that result, one finds that the generating series $\Theta_{X}$
of the primitive nef and big classes begins with 

\[
\begin{array}{c}
T^{4}+6(T^{10}+T^{12}+T^{18}+2T^{22}+T^{28}+3T^{30}+T^{34}+T^{36}+2T^{46}+3T^{58}\\
+2T^{66}+2T^{70}+T^{76}+3T^{82}+2T^{90}+2T^{94}+T^{100})...
\end{array}
\]

\subsection{Surfaces with Néron-Severi lattice of type $S_{3}$}

Let us study the K3 surfaces $X$ such that the Néron-Severi lattice
is generated by $L,A_{1},A_{2}$ with $L$ nef and $A_{1},A_{2}$
two $(-2)$-curves with intersection matrix 
\[
\left(\begin{array}{ccc}
12 & 0 & 0\\
0 & -2 & 1\\
0 & 1 & -2
\end{array}\right).
\]
There are exactly $4$ $(-2)$-curves on $X$: 
\[
A_{1},A_{2},A_{3}=L-3A_{1}-2A_{2},\,A_{4}=L-2A_{1}-3A_{2},
\]
their intersection matrix is
\[
\left(\begin{array}{cccc}
-2 & 1 & 4 & 1\\
1 & -2 & 1 & 4\\
4 & 1 & -2 & 1\\
1 & 4 & 1 & -2
\end{array}\right).
\]
The divisor 
\[
D=L-2A_{1}-2A_{2}
\]
satisfies 
\[
D^{2}=4
\]
 and it is ample since moreover 
\[
\forall i\in\{1,...,4\},\,\,A_{i}D=2>0.
\]
There are no big and nef divisors $D'$ with $D'^{2}<4$. 
\begin{prop}
The linear system $|D|$ defines an embedding $\phi_{D}:X\hookrightarrow\PP^{3}$.
The images of the four $(-2)$-curves by $\phi_{D}$ are $4$ conics.
Since 
\[
A_{1}+A_{3}=D,\,\,\,A_{2}+A_{4}=D,
\]
 the pairs of conics $(A_{1},A_{3})$ and $(A_{2},A_{4})$ are on
the same hyperplanes $P_{1},P_{2}$ respectively.
\end{prop}

\begin{proof}
There are no elliptic fibrations, so that by Theorem \ref{thm:SaintDonat},
the divisor $D$ defines an embedding. The remaining affirmations
are immediately checked.
\end{proof}
By \cite[Table 1]{AST}, a quartic surface that is invariant by an
involution fixing an isolated point and a plane has Picard number
at least $8$. Suppose that the quadric $F\hookrightarrow\PP^{3}$
is invariant by a involution $\iota:x\to(-x_{1}:-x_{2}:x_{3}:x_{4})$.
Since the Picard number of $X$ is $3$, the involution must be non-symplectic.
But then at least one of the fixed lines $x_{1}=x_{2}=0$ or $x_{3}=x_{4}=0$
must be contained in $F$. However the rational curves on $F$ are
conics, thus that case is also impossible. Since $D$ is the unique
big and nef divisor with $D^{4}=4$, we conclude that:
\begin{cor}
The automorphism group of the K3 surface $X$ is trivial. 
\end{cor}

There are $4$ big and nef divisors of self-intersection $6$, which
are not ample. These divisors are 
\[
L-2A_{1}-A_{2},\,\,L-A_{1}-2A_{2},\,\,2L-5A_{1}-4A_{2},\,\,2L-6A_{1}-5A_{2}.
\]
Each of these contracts a different $(-2)$-curve. There are $4$
big and nef divisors of self-intersection $12$:
\[
D_{1}=L,\,D_{2}=3L-8A_{1}-4A_{2},\,D_{3}=5L-12A_{1}-4A_{2},\,D_{4}=3L-4A_{1}-8A_{2},
\]
each of them contracts two $(-2)$-curves and these are the vertices
of the fundamental polygon $\mathcal{F}_{S_{3}}$. Since $DD_{i}=12$,
the maximum distance from $D$ and points in the polygon is $d_{max}=\arccosh(\sqrt{3}).$
The generating series $\Theta_{X}(T)$ of the primitive nef and big
classes begins with 

\[
\begin{array}{c}
T^{4}+4(T^{6}+T^{10}+T^{12}+T^{22}+T^{30}+2T^{34}+2T^{42}+2T^{46}+T^{52}\\
+T^{58}+2T^{66}+2T^{70}+T^{76}+T^{78}+2T^{82}+2T^{84}+4T^{94}+T^{100})...
\end{array}
\]

\subsection{Surfaces with Néron-Severi lattice of type $S_{4}$}

Let us study the K3 surfaces $X$ with Picard number $3$ and Néron-Severi
lattice generated by $D,A_{1},A_{2}$ with $D$ nef and $A_{1},A_{2}$
two $(-2)$-curves with Gram matrix 
\[
\left(\begin{array}{ccc}
2 & 1 & 2\\
1 & -2 & 1\\
2 & 1 & -2
\end{array}\right).
\]
Then the K3 surface contains exactly four $(-2)$-curves, with coordinates
\[
A_{1},A_{2},A_{3}=D-A_{1},A_{4}=D-A_{2},
\]
their intersection matrix is 
\[
\left(\begin{array}{cccc}
-2 & 1 & 3 & 1\\
1 & -2 & 1 & 6\\
3 & 1 & -2 & 1\\
1 & 6 & 1 & -2
\end{array}\right).
\]
The curves $A_{1},A_{2},A_{3}$ generate the Néron-Severi lattice
and one has $A_{4}=2A_{1}-A_{2}+2A_{3}.$ The divisor $D=A_{1}+A_{3}$
is ample with $D^{2}=2$. One has 
\[
DA_{1}=DA_{3}=1,\,\,DA_{2}=DA_{4}=2.
\]
Since there is no elliptic pencil, the linear system $|D|$ is base
point free and $X$ is a double cover of $\PP^{2}$ branched over
a smooth sextic curve $C_{6}$. The divisor $A_{1}+A_{3}$ is the
pull-back of a tritangent line to $C_{6}$. Since moreover 
\[
A_{2}+A_{4}=2D,
\]
the divisor $A_{2}+A_{4}$ is the pull-back of a conic tangent to
$C_{6}$ at $6$ points.Thus:
\begin{prop}
The surface $X$ is a double cover of $\PP^{2}$ branched over a smooth
sextic curve which admits a tritangent line and a $6$-tangent conic.
The $(-2)$-curves on $X$ are the pull-back of the tritangent line
and the conic. \\
The automorphism group of $X$ is $\ZZ/2\ZZ$.
\end{prop}

Since $A_{1}+A_{3}$ and $A_{2}+A_{4}$ are pull-back of curves in
$\PP^{2}$, the involution of the double cover maps $A_{1},A_{2},A_{3},A_{4}$
to $A_{3},A_{4},A_{1},A_{2}$. Each of the following big and nef
divisors 
\[
\begin{array}{c}
D_{1}=7A_{1}+5A_{2}+3A_{3},\,D_{2}=3A_{2}+5A_{2}+7A_{3}\\
D_{3}=13A_{1}-5A_{2}+17A_{3},\,D_{4}=17A_{1}-5A_{2}+13A_{3}.
\end{array}
\]
contracts two $(-2)$-curves $A_{s},A_{t}$ such that $A_{s}A_{t}=1$,
thus they are the vertices of the fundamental polygon $\mathcal{F}_{S_{4}}$.
We have $D_{i}^{2}=60$ and $DD_{i}=20$. The maximum distance from
$D$ to a point in $\mathcal{F}_{S_{4}}$ is
\[
d_{max}=\arccosh\left(\frac{(DD_{1})^{2}}{D^{2}D_{1}^{2}}\right)^{\frac{1}{2}}=\arccosh\left(\sqrt{\frac{10}{3}}\right).
\]
The generating series $\Theta_{X}(T)$ of the primitive nef and big
classes begins with

\[
\begin{array}{c}
T^{2}+2(T^{4}+2T^{6}+T^{10}+T^{14}+2T^{18}+3T^{22}+2T^{26}+T^{28}+3T^{34}+2T^{36}\\
+T^{38}+2T^{42}+5T^{46}+2T^{52}+4T^{54}+5T^{58}+2T^{60}+2T^{62}+2T^{66}+T^{68}\\
+4T^{74}+4T^{78}+4T^{82}+2T^{84}+4T^{86}+2T^{90}+T^{92}+5T^{94}+3T^{98})+...
\end{array}
\]
The two square $4$ divisors are $A_{1}+A_{2}+A_{3},A_{1}+A_{3}+A_{4}$,
they are orthogonal to respectively $A_{2},A_{4}$, and are exchanged
by the covering involution. 

\subsection{Surfaces with Néron-Severi lattice of type $S_{5}$\label{subsec:N=0000E9ron-Severi-S5}}

Let us study K3 surfaces such that their Néron-Severi lattice is
generated by a divisor nef $L$ and by two $(-2)$-curves $A_{1},A_{2}$
with intersection matrix 
\begin{equation}
\left(\begin{array}{ccc}
4 & 0 & 0\\
0 & -2 & 1\\
0 & 1 & -2
\end{array}\right).\label{eq:matrixS5}
\end{equation}
It contains exactly four $(-2)$-curves $A_{1},A_{2},A_{3},A_{4}$
with intersection matrix 
\[
M(S_{5})=\left(\begin{array}{cccc}
-2 & 1 & 3 & 0\\
1 & -2 & 0 & 3\\
3 & 0 & -2 & 1\\
0 & 3 & 1 & -2
\end{array}\right)
\]
and generating $\NS X)$. Their classes are
\[
A_{1},\,\,\,A_{2},\,\,\,A_{3}=L-2A_{1}-A_{2},\,\,\,A_{4}=L-A_{1}-2A_{2}=A_{1}-A_{2}+A_{3}.
\]
The divisor $D=L-A_{1}-A_{2}$ satisfies 
\[
D=A_{1}+A_{3}=A_{2}+A_{4}
\]
and
\[
D^{2}=2
\]
 and for each $(-2)$-curve $A_{i}$ we have $DA_{i}=1,$ thus it
is ample. Since $X$ has no elliptic fibration, the net $|D|$ is
free and the morphism $\phi_{D}$ associated to $|D|$ is a double
cover of the plane. Let $C_{6}$ be the sextic branch curve in $\PP^{2}$.
It is a smooth curve such that there are two lines $L_{1},L_{2}$
tangent to it at $3$ points each and the pull-back of $L_{1}$ is
$\pi^{*}L_{1}=A_{1}+A_{3}$, the pull-back of $L_{2}$ is $\pi^{*}L_{2}=A_{2}+A_{4}$
(so that $A_{1}A_{3}=3=A_{2}A_{4}$).  We thus have:
\begin{prop}
The surface $X$ is a double cover of the plane branched over a smooth
sextic curve which has two tritangent lines.\\
The automorphism group of such a surface is $\ZZ/2\ZZ$.
\end{prop}

The involution exchanges $A_{1}$ with $A_{3}$ and $A_{2}$ with
$A_{4}$. We construct the family of such sextic curves in Section
\ref{sec:Examples}. We have $L^{2}=4$ and since the surface contains
no elliptic curves, the linear system $|L|$ is base-point free and
not hyperelliptic. Since $LA_{1}=LA_{2}=0$ with $A_{1}A_{2}=1$,
the surface $X$ is the minimal resolution of a quartic surface with
one cusp; the images of the curves $A_{3}$ and $A_{4}$ have degree
$4$. Also the divisor $L'=3L-4(A_{1}+A_{2})$ is nef, and $L'^{2}=4.$
It plays a similar role for $A_{3},A_{4}$ as $L$ for $A_{1},A_{2}$.
We have $LD=L'D=4$ \\
The divisors $M=4L-A_{1}+3A_{2}$ and $M'=2L+A_{1}+3A_{2}$ are nef
with $M^{2}=M'^{2}=12$ and $DM=DM'=6$. They satisfy $MA_{1}=MA_{4}=0$
and $M'A_{2}=M'A_{3}=0$. \\
The divisors $L,L',M,M'$ form the $4$ vertices of the fundamental
polygon $\mathcal{F}_{S_{5}}.$ The maximum distance from $D$ to
a points in $\mathcal{F}_{S_{5}}$ is 
\[
d_{max}=\arccosh(\max(\frac{(DL)^{2}}{D^{2}L^{2}},\frac{(DM)^{2}}{D^{2}M^{2}})^{1/2})=\arccosh(\sqrt{2}).
\]
The generating series $\Theta_{X}(T)$ of the primitive nef and big
classes on a K3 surface $X$ with Néron-Severi lattice isometric to
$S_{5}$ begins with

\[
\begin{array}{c}
T^{2}+2(T^{4}+2T^{10}+T^{12}+T^{14}+2T^{22}+T^{26}+T^{28}+2T^{30}+T^{34}\\
+2T^{38}+T^{44}+3T^{46}+2T^{50}+4T^{58}+3T^{62}+T^{66}+T^{68}+2T^{70}\\
+3T^{74}+2T^{76}+3T^{82}+2T^{86}+2T^{92}+3T^{94}+3T^{98}+2T^{100})+...
\end{array}
\]
The two nef divisors of square $4$ are orthogonal to $A_{1},A_{2}$
and $A_{3},A_{4}$ respectively. 

\subsection{Surfaces with Néron-Severi lattice of type $S_{6}$\label{subsec:Surfaces-with-N=0000E9ron-SeveriS6}}

Let us study the K3 surfaces $X$ with Picard number $3$ and generated
by a nef divisor $D$ and $(-2)$-curves $A_{1},A_{3}$ with Gram
matrix
\[
\left(\begin{array}{ccc}
2 & 2 & 3\\
2 & -2 & 1\\
3 & 1 & -2
\end{array}\right).
\]
The K3 surface contains exactly six $(-2)$-curves $A_{1},\dots,A_{6}$
with intersection matrix 
\[
\left(\begin{array}{cccccc}
-2 & 6 & 1 & 5 & 5 & 1\\
6 & -2 & 5 & 1 & 1 & 5\\
1 & 5 & -2 & 11 & 0 & 9\\
5 & 1 & 11 & -2 & 9 & 0\\
5 & 1 & 0 & 9 & -2 & 11\\
1 & 5 & 9 & 0 & 11 & -2
\end{array}\right).
\]
The $(-2)$-curves $A_{1},A_{3},A_{5}$ form a basis of $\NS X)$
and 
\[
A_{2}=A_{1}-2A_{3}+2A_{5},\,\,\,A_{4}=3A_{1}-4A_{3}+3A_{5},\,\,\,A_{6}=3A_{1}-3A_{3}+2A_{5}.
\]
The divisor 
\[
D=A_{1}-A_{3}+A_{5}
\]
is the unique big and nef divisor such that $D^{2}=2.$ We have $DA_{1}=DA_{2}=2$
and moreover $DA_{j}=3$ for $j\in\{3,...,6\}$, thus $D$ is ample.
We observe moreover that 
\begin{equation}
\begin{array}{c}
A_{1}+A_{2}=2D\\
A_{3}+A_{4}=3D\\
A_{5}+A_{6}=3D.
\end{array}\label{eq:S6}
\end{equation}
Since $X$ has no elliptic fibration, the net $|D|$ is free and the
morphism $\phi_{D}$ associated to $|D|$ is a double cover of the
plane. We have:
\begin{prop}
The branch locus $C_{6}$ of $\phi_{D}$ is a smooth sextic curve,
there exists a $6$-tangent conic $C_{o}$ to $C_{6}$ and there are
two cuspidal cubics $E_{1},E_{2}$ such that the pull-back to $X$
of $C_{o},E_{1},E_{2}$ are respectively $A_{1}+A_{2},\,A_{3}+A_{4},\,A_{5}+A_{6}$.
\\
The tangent to the cusp $p_{i}$ of $E_{i}$ is distinct from the
tangent to the sextic curve at $p_{i}$. The cubic $E_{i}$ cuts $C_{6}$
tangentially at $8$ other points than $p_{i}$.\\
The automorphism group of $X$ is $\ZZ/2\ZZ$.
\end{prop}

From Equation \eqref{eq:S6}, the action of the involution on the
Néron-Severi lattice is by exchanging $A_{1},A_{3},A_{5}$ with respectively
$A_{2},A_{4},A_{6}$. 
\begin{proof}
It remains to prove the assertions on singularities of the cubics.
The cubic curve cannot be nodal, otherwise there would be some ramification
in that node and the pull-back would be irreducible. If the tangent
to the cusp was the tangent to the sextic $C_{6}$, then that would
also create some ramification on the curve $E_{i}$. \\
By performing $4$ successive blow-ups at $p_{i}$, taking the double
cover and going to the minimal model $X$, we obtain that the local
intersection number above $p_{i}$ of the two $(-2)$-curves above
$E_{i}$ is equal to $3$. Since there must be no ramification over
$E_{i}$, the cubic $E_{i}$ cuts $C$ tangentially at $8$ other
points than $p_{i}$.
\end{proof}
In base $A_{1},A_{3},A_{5}$ the following $4$ divisors 
\[
D_{1}=(34,-49,41),\,D_{2}=(10,5,3),\,D_{3}=(34,-27,19),\,D_{4}=(10,-17,25)
\]
are nef and have square $D_{i}^{2}=132.$ Moreover $D_{1}A_{2}=D_{1}A_{4}=0$,
$D_{2}A_{1}=D_{2}A_{3}=0$, $D_{3}A_{1}=D_{3}A_{6}=0$, $D_{4}A_{2}=D_{4}A_{5}=0$.
These divisors are $4$ vertices of the fundamental domain, which
is an hexagon. The divisors 
\[
D_{5}=(2,1,5),\,D_{6}=(20,-23,17)
\]
are nef and $D_{5}^{2}=D_{6}^{2}=44$, $DD_{5}=DD_{6}=22$. Moreover
$D_{5}A_{4}=D_{5}A_{6}=0=D_{6}A_{3}=D_{6}A_{5}$, so that $D_{5}$
and $D_{6}$ are the remaining vertices of the hexagon. The maximum
distance between $D$ and points in $\mathcal{F}_{S_{6}}$ is 
\[
\arccosh(\max(\frac{44^{2}}{2\cdot132},\frac{22^{2}}{2\cdot44})^{1/2})=\arccosh\left(\sqrt{\frac{22}{3}}\right).
\]
The generating series $\Theta_{X}(T)$ of the primitive nef and big
classes begins with

\[
\begin{array}{c}
T^{2}+2(T^{4}+2T^{6}+3T^{10}+2T^{12}+3T^{14}+2T^{18}+4T^{26}+2T^{28}+2T^{30}+7T^{34}\\
+2T^{36}+2T^{38}+4T^{42}+T^{44}+7T^{46}+3T^{50}+6T^{54}+6T^{58}+2T^{60}+3T^{62}\\
+T^{68}+8T^{70}+3T^{74}+5T^{76}+6T^{78}+7T^{82}+4T^{86}+6T^{90}+T^{92}+11T^{94})+...
\end{array}
\]
The two nef divisors of square $4$ are orthogonal to $A_{1}$ and
$A_{2}$ respectively. 

\section{Rank $4$ and compact cases}

According to Vinberg's classification \cite[Theorem 1]{Vinberg},
there are only two lattices of rank $4$ with compact fundamental
domain.

\subsection{\label{subsec:Lattice-L24}Lattice of rank $4$ of type $L(24)$}

Let us consider a $K3$ surface whose Néron-Severi lattice is the
lattice denoted $L(24)$ in \cite{Vinberg}. That lattice is generated
by a divisor $L$ and $(-2)$-curves $A_{1},A_{2},A_{3}$ with intersection
matrix
\[
\left(\begin{array}{cccc}
2 & 1 & 1 & 1\\
1 & -2 & 0 & 0\\
1 & 0 & -2 & 0\\
1 & 0 & 0 & -2
\end{array}\right).
\]
The surface $X$ contains only three other $(-2)$-curves 
\[
A_{4}=L-A_{1},\,\,A_{5}=L-A_{2},\,\,A_{6}=L-A_{3},
\]
and $A_{1},\dots,A_{4}$ is a basis of $\NS X)$. The intersection
matrix of the $6$ curves is 
\[
\left(\begin{array}{cccccc}
-2 & 0 & 0 & 3 & 1 & 1\\
0 & -2 & 0 & 1 & 3 & 1\\
0 & 0 & -2 & 1 & 1 & 3\\
3 & 1 & 1 & -2 & 0 & 0\\
1 & 3 & 1 & 0 & -2 & 0\\
1 & 1 & 3 & 0 & 0 & -2
\end{array}\right).
\]
The divisor $L$ is ample. Since $L^{2}=2$ and there are no elliptic
curves on $X$, the linear system $|L|$ defines a double cover $\phi_{L}:X\to\PP^{2}$
of the plane $\PP^{2}$. One checks that $LA_{i}=1$, and in fact
\begin{equation}
L=A_{1}+A_{4}=A_{2}+A_{5}=A_{3}+A_{6},\label{eq:L24}
\end{equation}
so that:
\begin{prop}
The surface $X$ is a double cover of $\PP^{2}$. The sextic branch
curve is smooth, has exactly $3$ tritangent lines $L_{1},L_{2},L_{3}$
such that their pull-back to $X$ are the six $(-2)$-curves on $X$.
The automorphism group of $X$ is $\ZZ/2\ZZ$.
\end{prop}

From Equation \eqref{eq:L24}, the action of the involution on the
Néron-Severi lattice is by exchanging $A_{1},A_{2},A_{3}$ with respectively
$A_{4},A_{5},A_{6}$. 

There are $20$ triple of $(-2)$-curves; the orthogonal complement
of a triple $A_{i},A_{j},A_{k}$ is generated by a big and nef divisor
$D_{i,j,k}$ for only $8$ such triples, which are 
\[
\begin{array}{cc}
D_{123}=2L+A_{1}+A_{2}+A_{3} & D_{156}=8L+4A_{1}-3A_{2}-3A_{3}\\
D_{456}=5L-A_{1}-A_{2}-A_{3} & D_{234}=6L-4A_{1}+3A_{2}+3A_{3}\\
D_{126}=6L+3A_{2}+3A_{3}-4A_{4} & D_{246}=8L-3A_{1}+4A_{2}-3A_{3}\\
D_{135}=6L+3A_{2}-4A_{3}+3A_{4} & D_{345}=8L-3A_{1}-3A_{2}+4A_{3}
\end{array}.
\]
These classes are the vertices of the fundamental polygon $\mathcal{F}_{X}$.
The divisors $D_{123}$ and $D_{345}$ have self-intersection $14$
and degree $7$ with respect to $D$. The $6$ other have self-intersection
$28$ and degree $14$. The maximum distance between $D$ and points
in $\mathcal{F}_{X}$ is 
\[
\ell=\arccosh(\max(\frac{7^{2}}{2\cdot14},\frac{14^{2}}{2\cdot28})^{1/2})=\arccosh(\sqrt{\frac{7}{2}}).
\]
The generating series $\Theta_{X}(T)$ of the primitive nef and big
classes begins with

\[
\begin{array}{c}
T^{2}+2(3T^{4}+3T^{6}+3T^{10}+3T^{12}+4T^{14}+3T^{18}+3T^{20}+9T^{22}+6T^{26}+6T^{28}\\
+7T^{30}+6T^{34}+6T^{36}+9T^{38}+6T^{42}+9T^{44}+18T^{46}+10T^{50}+9T^{52}+12T^{54}\\
+18T^{58}+9T^{60}+15T^{62}+9T^{66}+9T^{68}+12T^{70}+22T^{74}+15T^{76}+15T^{78}+15T^{82})...
\end{array}
\]
The $6$ nef divisors of square $4$ are respectively orthogonal to
the $6$ pairs of curves $A_{i},A_{j}$ such that $A_{i}A_{j}=1$. 

\subsection{Lattice of rank $4$ of type $L(27)$\label{subsec:Lattice-L27}}

Consider the rank $4$ lattice denoted $L(27)$ in \cite{Vinberg},
with intersection matrix 
\[
\left(\begin{array}{cccc}
12 & 2 & 0 & 0\\
2 & -2 & 1 & 0\\
0 & 1 & -2 & 1\\
0 & 0 & 1 & -2
\end{array}\right).
\]
This is a lattice of some K3 surface $X$ that possess only $8$ $(-2)$-curves
$A_{1},\dots,A_{8}$. The intersection matrix of these curves is
\[
\left(\begin{array}{cccccccc}
-2 & 3 & 1 & 1 & 1 & 1 & 1 & 1\\
3 & -2 & 1 & 1 & 1 & 1 & 1 & 1\\
1 & 1 & -2 & 6 & 4 & 0 & 4 & 0\\
1 & 1 & 6 & -2 & 0 & 4 & 0 & 4\\
1 & 1 & 4 & 0 & -2 & 6 & 4 & 0\\
1 & 1 & 0 & 4 & 6 & -2 & 0 & 4\\
1 & 1 & 4 & 0 & 4 & 0 & -2 & 6\\
1 & 1 & 0 & 4 & 0 & 4 & 6 & -2
\end{array}\right),
\]
and the curves $A_{2},A_{4},A_{5}$ and $A_{7}$ are generators of
Néron-Severi lattice $L(27)$. The surface $X$ contains a unique
big and nef divisor $D$ with $D^{2}=2$, which is $D=A_{1}+A_{2}$.
We have 
\[
DA_{1}=DA_{2}=1,DA_{j}=2\,\text{for }j\geq2,
\]
thus $D$ is ample. One can check that 
\begin{equation}
A_{1}+A_{2}=D,\,\,A_{3}+A_{4}=A_{5}+A_{6}=A_{7}+A_{8}=2D.\label{eq:L27}
\end{equation}
 The linear system $|D|$ is base point free and 
\begin{prop}
The surface $X$ is a double cover of $\PP^{2}$. The sextic branch
curve is smooth, has exactly one tritangent line $L_{0}$ and three
$6$-tangent conics $C_{1},C_{2},C_{3}$ such that the pull-back to
$X$ of $L_{0},C_{1},C_{2},C_{3}$ are the eight $(-2)$-curves on
$X$. The automorphism group of $X$ is $\ZZ/2\ZZ$.
\end{prop}

From Equation \eqref{eq:L27}, the action of the involution on the
Néron-Severi lattice is by exchanging $A_{1},A_{3},A_{5},A_{7}$ with
respectively $A_{2},A_{4},A_{6},A_{8}$. The fundamental domain has
$12$ vertices $D_{1},\dots,D_{12}$, whose coordinates in base $A_{2},A_{4},A_{5}$,
$A_{7}$ are:

\[
\begin{array}{ccc}
(6,3,7,2) & (6,-27,22,17) & (-6,-18,23,13)\\
(6,3,2,7) & (6,-27,17,22) & (-6,-18,13,23)\\
(6,-12,17,7) & (-6,-3,13,8) & (-6,-33,28,23)\\
(6,-12,7,17) & (-6,-3,8,13) & (-6,-33,23,28).
\end{array}
\]
We have $D_{i}^{2}=60$ and $D_{i}D=30$, thus the maximum distance
between $D$ and points in $\mathcal{F}_{X}$ is $\arccosh(\sqrt{\frac{15}{2}})$.
The generating series for the big and nef divisors begins with

\[
\begin{array}{c}
T^{2}+2(3T^{4}+9T^{6}+7T^{10}+12T^{12}+9T^{14}+6T^{18}+3T^{20}+13T^{22}+12T^{26}+9T^{28}\\
+12T^{30}+27T^{34}+18T^{36}+16T^{38}+24T^{42}+12T^{44}+39T^{46}+15T^{50}+18T^{52}\\
+36T^{54}+31T^{58}+27T^{60}+22T^{62}+30T^{66}+15T^{68}+36T^{70}+30T^{74}+36T^{76})...
\end{array}
\]
The $6$ nef divisors of square $4$ are respectively orthogonal to
the $6$ curves $A_{3},\dots,A_{8}$. 

\section{Constructions of surfaces\label{sec:Examples}}

\subsection{Construction of surfaces of type $S_{1}$\label{subsec:ExamplesS1}}

Our previous results tell us that a surface of type $S_{1}$ is the
double cover of $\PP^{2}$ branched over a smooth sextic which has
three $6$-tangent conics. Conversely:
\begin{prop}
\label{prop:Picard3NeronS1}Let $\pi:Y\to\PP^{2}$ be the double cover
of $\PP^{2}$ branched over a smooth sextic $C_{6}$ which has three
$6$-tangent smooth conics $C_{1},C_{2},C_{3}$. Suppose that $Y$
has Picard number $3$. Then the Néron-Severi lattice of $Y$ is isometric
to $S_{1}$ and $Y$ contains only $6$ $(-2)$-curves.
\end{prop}

\begin{proof}
The pull-back of the conics splits as $\pi^{*}C_{k}=B_{2k-1}+B_{2k}$
(for $k\in\{1,2,3\}$) with $B_{1}B_{2}=B_{3}B_{4}=B_{5}B_{6}=6$.
The involution coming from the double plane structure exchanges the
pairs of curves $(B_{1},B_{2})$, $(B_{3},B_{4})$ and $(B_{5},B_{6})$,
thus 
\[
B_{1}(B_{3}+B_{4})=B_{2}(B_{3}+B_{4})=\frac{1}{2}\pi^{*}C_{1}\pi^{*}C_{2}=4,
\]
and also $B_{3}(B_{1}+B_{2})=B_{4}(B_{1}+B_{2})=4$. Let us define
$B_{1}B_{3}=a$ (with $a\in\{0,\dots,4\}$). Then $B_{1}B_{4}=4-a$,
$B_{2}B_{3}=4-a$, $B_{2}B_{4}=a.$ Continuing in the same way for
the other intersections, one sees that there exist $a,b,c\in\{0,1,...,4\}$
such that 
\[
(B_{i}B_{j})_{1\leq i,j\leq6}=\left(\begin{array}{cccccc}
-2 & 6 & a & \bar{a} & b & \bar{b}\\
6 & -2 & \bar{a} & a & \bar{b} & b\\
a & \bar{a} & -2 & 6 & c & \bar{c}\\
\bar{a} & a & 6 & -2 & \bar{c} & c\\
b & \bar{b} & c & \bar{c} & -2 & 6\\
\bar{b} & b & \bar{c} & c & 6 & -2
\end{array}\right),
\]
where $\bar{x}=4-x$. Up to exchanging $B_{3}$ with $B_{4}$ and
$B_{5}$ with $B_{6}$, we can suppose $a,b\in\{0,1,2\}$. By computing
all minors of size $3$, we obtain that the above intersection matrix
has rank at least $3$ (and at most $4$), moreover it has rank $3$
if and only if $a=b=0$ and $c=4$. \\
If this is the case, then the Néron-Severi lattice of the K3 surface
$X$ contains a sub-lattice isometric to $S_{1}$, generated by $\frac{1}{2}(B_{1}+B_{2}),B_{1},B_{3}$.
Since the discriminant group of $S_{1}$ does not contains isotropic
elements, we conclude that $\NS X)$ is isometric to $S_{1}$. 
\end{proof}
An example of a K3 surface with three $6$-tangent conics has been
constructed by Elsenhans and Jähnel as an auxiliary surface to obtain
double plane K3 surfaces with Picard number $1$ (see \cite[3.4.1]{EJ}).
Their idea is to consider a sextic curve $C_{6}$ with equation $f_{6}$
of the form 
\[
f_{6}=q_{1}q_{2}q_{3}-f_{3}^{2},
\]
where the $q_{j}$ are quadrics and $f_{3}$ is a cubic. If the forms
$q_{1},q_{2},q_{3}$ and $f_{3}$ are generic, then the curve $f_{6}=0$
is smooth and the intersection points of the cubic curve $F_{3}=\{f_{3}=0\}$
and the quadric $Q_{j}=\{q_{i}=0\}$ are transverse. For such point
$p$ of intersection, let $u$ be a local parameter for $Q_{j}$ and
$v$ a local parameter for $F_{3}$. From the equation $f_{6}$, the
parameter $u$ is also a local parameter for the sextic, thus the
local intersection multiplicity of $C_{6}$ and $Q_{j}$ is $\geq2$
at $p$. Since there are $6$ such points on $Q_{j}$, the quadric
$Q_{j}$ is $6$-tangent to the sextic curve $C_{6}$. In fact
\begin{thm}
\label{thm:S1-equation-of-C6}Let $C_{6}$ be a generic smooth sextic
curve which possesses three $6$-tangent conics $C_{j}=\{q_{j}=0\}$,
$j\in\{1,2,3\}$. Then there exists a cubic form $f_{3}$ such that
$C_{6}=\{q_{1}q_{2}q_{3}-f_{3}^{2}=0\}.$
\end{thm}

\begin{rem*}
Another proof of that result is given in \cite[Theorem 3.6]{ACL}.
\end{rem*}
\begin{proof}
Let $C_{6}=\{q_{1}q_{2}q_{3}-f_{3}^{2}=0\}$ be a generic sextic curve
as above and let $Y\to\PP^{2}$ be the double cover branched over
$C_{6}$. Above each irreducible conics $q_{i}=0,\,(1\leq i\leq3)$,
there are two $\cu$-curves $B_{2i-1}+B_{2i}$ such that $B_{2i-1}B_{2i}=6$.
On the example of Proposition \ref{prop:The-surface-3conics} below,
the intersection matrix $(B_{i}B_{j})_{1\leq i,j\leq6}$ is (up to
permutation of $B_{2i-1},B_{2i}$) the rank $3$ intersection matrix
of the $6$ $\cu$-curves on a K3 surface with Néron-Severi lattice
isomorphic to $S_{1}$. These intersection numbers remain the same
for the $\cu$-curves in that flat family of smooth surfaces $Y$.
The Picard number of $Y$ is at least $3$ and if the equation of
$C_{6}$ is moreover generic, the Picard number of $Y$ is $3$ and
therefore by Proposition \ref{prop:Picard3NeronS1}, one has $\NS Y)\simeq S_{1}$. 

Let $\Phi$ be the map
\[
\Phi:\left|\begin{array}{ccc}
H^{0}(\PP^{2},\OO(2))^{\oplus3}\oplus H^{0}(\PP^{2},\OO(3)) & \to & H^{0}(\PP^{2},\OO(6))\\
\tau=(q_{1},q_{2},q_{3},f_{3}) & \to & f_{\tau}=q_{1}q_{2}q_{3}-f_{3}^{2}
\end{array}\right..
\]
That map is invariant under the action of 
\[
\G:(q_{1},q_{2},q_{3},f_{3})\to(\l q_{1},\mu q_{2},\frac{1}{\l\mu}q_{3},f_{3})
\]
for $\l,\mu\in\CC^{*}$. Suppose that $f_{\tau}=f_{\tau'}$ for $\tau'=(q_{1}',q_{2}',q_{3}',f_{3}')$
and the form $\tau$ is generic so that the double cover $Y$ branched
over the sextic $f_{\tau}=0$ has $\NS Y)\simeq S_{1}$. Since the
conics $q_{i}=0$ are the images of the $6$ $\cu$-curves on $Y$,
up to rescaling by using a transformation $\G$ and up to taking a
permutation of the $q_{j}$'s, one can suppose that , $q'_{1}=q_{1}$,
$q'_{2}=q_{2}$, and $q'_{3}=\nu q_{3}$ for a scalar $\nu$. Then
one obtains the relations
\[
q_{1}q_{2}q_{3}(1-\nu)=(f_{3}-f_{3}')(f_{3}+f_{3}').
\]
If $\nu\neq1$, then the irreducible degree $2$ form $q_{1}$ divides
one of the two degree $3$ factors $(f_{3}-f_{3}'),\,(f_{3}+f_{3}')$.
But then $q_{2}$ or $q_{3}$ cannot be irreducible, which is a contradiction.
Therefore $\nu=1$ and $f_{3}=\pm f_{3}'$. Thus the dimension of
the (unirational) moduli of sextic curves with an equation of the
form $q_{1}q_{2}q_{3}-f_{3}^{2}=0$ is $6\cdot3+10-2-9=17$ dimensional.
\end{proof}
In \cite[Section 2]{Roulleau}, we proved that the moduli spaces of
K3 surfaces with finite automorphism group and Picard number $\geq3$
are irreducible. From the above discussion, we get:
\begin{cor}
The moduli space of K3 surfaces of type $S_{1}$ is unirational.
\end{cor}

\begin{proof}
By Theorem \ref{thm:S1-equation-of-C6}, the rational map from $H^{0}(\PP^{2},\OO(2))^{\oplus3}\oplus H^{0}(\PP^{2},\OO(3))$
to the moduli space of $S_{1}$-polarized K3 surfaces defined by associating
to $(q_{1},q_{2},q_{3},f_{3})$ the double cover surface branched
over the sextic $\{q_{1}q_{2}q_{3}-f_{3}^{2}=0\}$ is dominant. 
\end{proof}
In order to complete the proof of Theorem \ref{thm:S1-equation-of-C6},
let us choose the following equations:
\[
\begin{array}{cc}
q_{1}=2(xy+y^{2}+xz),\hfill & q_{2}=4(x^{2}+xy+z^{2}),\hfill\\
q_{3}=5(x^{2}+xy+yz+z^{2}), & f_{3}=x^{2}y+y^{3}+y^{2}z+xz^{2}+z^{3}.
\end{array}
\]
Then the associated sextic $C_{6}$ is smooth and the three conics
$Q_{j}=\{q_{j}=0\}$ are $6$-tangent to it. 
\begin{prop}
\label{prop:The-surface-3conics}The surface $Y$ which is the double
cover of the plane branched over the curve $C_{6}=\{q_{1}q_{2}q_{3}-f_{3}^{2}=0\}$
has Picard number $3$ and Néron-Severi lattice $S_{1}$.
\end{prop}

For computing the Picard number, we use the method of van Luijk \cite{vL}
refined by Elsenhans and Jahnel in \cite[Section 3.3.1]{EJ}. We first
need to recall:
\begin{thm}
(Artin-Tate Conjecture; satisfied for K3 surfaces). Let $Y$ be a
K3 surface over $\FF_{q}$. Denote by $\rho$ the rank and by $\Delta_{q}$
the discriminant of the Picard group of $Y_{/\FF_{q}}$. Then
\[
\lim_{t\to q}\frac{P_{W}(Y,t)}{(t-q)^{\rho}}=q^{21-\rho}\#Br(Y)|\Delta_{q}|,
\]
where $P_{W}(Y,t)$ is the second Weil polynomial and $Br(Y)$ denotes
the Brauer group of $Y_{/\FF_{q}}$.
\end{thm}

\begin{proof}
(Of Proposition \ref{prop:The-surface-3conics}). Let $Y_{17}$ and
$Y_{29}$ be the reduction modulo $17$ and $29$ respectively of
$Y$. Using Magma \cite{magma}, one finds that the Weil polynomial
of $Y_{29}$ is 

\[
\begin{array}{c}
P_{29}=(T-29)^{4}(T^{18}+34T^{17}+1508T^{16}+47096T^{15}+1487729T^{14}+36778612T^{13}\\
+594823321T^{12}+10706819778T^{11}-68999505236T^{10}-1000492825922T^{9}\\
-58028583903476T^{8}+7572730199403618T^{7}+353814783205469041T^{6}\\
+18398368726684390132T^{5}+625898351490474733529T^{4}\\
+16663261029844769954936T^{3}+448717814875105590929348T^{2}\\
+8508380105131809858775714T+210457284365172120330305161).
\end{array}
\]
In that product, the degree $18$ factor $Q_{18}$ is irreducible
and the roots of $Q_{18}(\frac{t}{29})$ are not roots of unity, thus
by the Tate conjecture (which is satisfied for K3 surfaces), the Picard
number of $Y_{29}$ is $4$ and this is also the geometric Picard
number. We obtain that $\lim_{t\to29}\frac{P_{29}(t)}{(t-29)^{4}}=2^{2}163\cdot29^{17}$.
Since the order of the Brauer group is a square, we see that the class
in $\QQ^{*}$ of the discriminant $\Delta_{29}$ modulo squares in
$\QQ^{*}$ equals $163$.\\
 The Weil polynomial of $Y_{17}$ is
\[
\begin{array}{c}
P_{17}=(T-17)^{4}(T^{18}+20T^{17}+323T^{16}+2890T^{15}+34391T^{14}+668168T^{13}\\
-1419857T^{12}+72412707T^{11}+4513725403T^{10}+139515148820T^{9}\\
+1304466641467T^{8}+6047981701347T^{7}-34271896307633T^{6}\\
+4660977897838088T^{5}+69332046230341559T^{4}\\
+1683778265594009290T^{3}+54386037978686500067T^{2}\\
+973223837513337369620t+14063084452067724991009).
\end{array}
\]
Again the degree $18$ factor $R_{18}$ is such that $R_{18}(\frac{t}{17})$
has no cyclotomic factor, thus the surface $Y_{17}$ has geometric
Picard number $4$. Since $\lim_{t\to17}\frac{P_{17}(t)}{(t-17)^{4}}=2^{4}13\cdot17^{17}$,
the class of the discriminant $\Delta_{17}$ modulo squares in $\QQ^{*}$
is $13$. If the surface $Y$ had Picard number $4$, the discriminants
$\Delta_{17}$ and $\Delta_{29}$ would belong in the same class modulo
a square (see \cite[Section 3.3.1]{EJ}), thus the surface $Y$ has
(geometric) Picard number $3$, and by Proposition \ref{prop:Picard3NeronS1},
it has Néron-Severi lattice $S_{1}$.
\end{proof}

\subsection{Construction of surfaces of type $S_{2}$\label{subsec:ExamplesS2}}

We recall that a surface of type $S_{2}$ is a smooth quartic surface
in $\PP^{3}$ with Picard number $3$ and three quadric sections,
each of which is the union of two degree $4$ smooth rational curves. 

Conversely, let $Y$ be a smooth quartic surface with three quadric
sections $Q_{1},Q_{2},Q_{3}$ such that each decomposes into the union
of two degree $4$ smooth rational curves : 
\[
Q_{1}=B_{1}+B_{2},\,Q_{2}=B_{3}+B_{4},\,Q_{3}=B_{5}+B_{6}.
\]
Because $Q_{i}^{2}=16$ and $B_{k}^{2}=-2$, we get $B_{1}B_{2}=B_{3}B_{4}=B_{5}B_{6}=10$. 
\begin{prop}
\label{prop:Suppose-S2}Suppose that the Picard number of $Y$ equals
to $3$. Then the Néron-Severi lattice of $Y$ is isometric to $S_{2}$.
\end{prop}

\begin{proof}
Because $Q_{s}Q_{t}=16$ for $1\leq s,t\leq3$, the intersection matrix
of the $6$ curves $B_{i},\,1\leq i\leq6$ has the form
\[
M=\left(\begin{array}{cccccc}
-2 & 10 & a_{1} & b_{1} & a_{2} & b_{2}\\
10 & -2 & c_{1} & d_{1} & c_{2} & d_{2}\\
a_{1} & c_{1} & -2 & 10 & a_{3} & b_{3}\\
b_{1} & d_{1} & 10 & -2 & c_{3} & d_{3}\\
a_{2} & c_{2} & a_{3} & c_{3} & -2 & 10\\
b_{2} & d_{2} & b_{3} & d_{3} & 10 & -2
\end{array}\right)
\]
where the $a_{i},b_{i},c_{i},d_{i}$ are integers in $\{0,1,...,16\}$
such that $a_{i}+b_{i}+c_{i}+d_{i}=16$. Up to exchanging $B_{3}$
and $B_{4}$ and $B_{5}$ and $B_{6}$, we can suppose that $a_{1}\leq b_{1}$
and $a_{2}\leq b_{2}$. Using a computer for checking the finite number
of possibilities, one finds that the only solution for $M$ to be
of rank $3$ (with $a_{1}\leq b_{1}$ and $a_{2}\leq b_{2}$) is
\[
(a_{1},b_{1},c_{1},d_{1},a_{2},b_{2},c_{2},d_{2},a_{3},b_{3},c_{3},d_{3})=(1,7,7,1,1,7,7,1,7,1,1,7).
\]
Then the Néron-Severi lattice of $Y$ contains $S_{2}$. The discriminant
group $A(S_{2})$ of $S_{2}$ is generated by classes $\frac{1}{36}L$
and $\frac{1}{3}(A_{1}-A_{2})$. The unique non-trivial isotropic
element in $A(S_{2})$ is $a=\frac{1}{3}L$ (one has $a^{2}=4=0\in\QQ/2\ZZ$).
The corresponding over-lattice is $S_{5}$, but a K3 surface with
Néron-Severi lattice $S_{5}$ does not have an ample class $D$ with
$D^{2}=4$, thus the Néron-Severi lattice of $Y$ is $S_{2}$. 
\end{proof}
\begin{prop}
The moduli space of K3 surfaces $X$ with $\NS X)\simeq S_{2}$ is
unirational.
\end{prop}

\begin{proof}
Consider two rational quartic curves $C_{1}$ and $C_{3}$ in $\PP^{3}$
such that the degree of their intersection scheme $C_{1}\cdot C_{3}$
is $1$ (such rational quartic curves are parametrized by a projective
space). Let $Y$ be a smooth quartic in the linear system of quartics
containing $C_{1}$ and $C_{3}$. By \cite[Exercise IV.6.1]{Hartshorne},
there exists a unique quadric $Q$ (resp. $Q'$) which is moreover
smooth, containing the curve $C_{1}$ (resp. $C_{3}$). The intersection
of $Y$ and $Q$ (resp. $Q'$) is the union of $C_{1}$ (resp. $C_{3}$)
and a degree $4$ curve $C_{2}$ (resp. $C_{4}$). Since 
\[
8=2HC_{1}=(C_{1}+C_{2})C_{1}
\]
(for a hyperplane $H$), we get $C_{1}C_{2}=10$, and therefore from
$(C_{1}+C_{2})^{2}=16$, we obtain $C_{2}^{2}=-2$. Similarly, the
intersection of $Y$ with $Q'$ is the union $C_{3}+C_{4}$ with $C_{3}C_{4}=10$
and $C_{4}^{2}=-2$. Since by construction $C_{1}C_{3}=1$, and moreover
\[
8=C_{1}(C_{3}+C_{4}),
\]
we get $C_{1}C_{4}=7$. From the construction, the curves $C_{1},\dots,C_{4}$
and $H$ generate a sub-lattice $S_{2}$ in the Néron-Severi lattice
of $Y$. The Picard number of $Y$ is therefore $\geq3$ and the generic
surface has $\NS Y)\simeq S_{2}$. From that construction, it is
clear that the moduli space is unirational.
\end{proof}
To show the effectiveness of the construction, let us give an example
of a K3 surface with Néron-Severi lattice $S_{2}$.
\begin{example}
\label{exa:COSTA_S2}Let $Y$ be the smooth quartic surface with equation
\[
\begin{array}{c}
-27x^{4}-246x^{3}y+102xy^{3}-216y^{4}-108x^{3}z+348x^{2}yz+459xy^{2}z+270y^{3}z\\
-162x^{2}z^{2}+51xyz^{2}-27y^{2}z^{2}-108xz^{3}-27z^{4}-9x^{3}t-9x^{2}yt+522xy^{2}t\\
-648y^{3}t+46x^{2}zt-369xyzt+911y^{2}zt-170xz^{2}t-279yz^{2}t-44z^{3}t\\
-27x^{2}t^{2}+672xyt^{2}-84xzt^{2}+546yzt^{2}+24z^{2}t^{2}+72xt^{3}=0.
\end{array}
\]
Each quadrics $Q_{1}=\{xy+y^{2}+2xz+yz+z^{2}+3xt+3yt\}$ and $Q_{2}=\{3xy-zt\}$
cuts $Y$ into two smooth quartic rational curves $B_{1},B_{2}$ and
$B_{3},B_{4}$ respectively. Moreover, one can check that $B_{1}B_{3}=B_{2}B_{4}=1$
and $B_{2}B_{3}=B_{1}B_{4}=7$, thus the curves $B_{1},\dots,B_{4}$
and the class of the hyperplane section generate a lattice of type
$S_{2}$. If the Picard number of the surface $Y$ equals $3$, then
$S_{2}$ is the Néron-Severi lattice of $Y$ (by the same argument
as in the proof of Proposition \ref{prop:Suppose-S2}) and therefore
there exists a third quadric $Q_{3}$ that cuts $Y$ into the union
of the two remaining $(-2)$-curves. Let $Y_{23}$ and $Y_{41}$ the
reductions modulo $23$ and $41$ of the surface $Y$. Thanks to the
computations \cite{Costa4} of Edgar Costa using algorithms developed
in \cite{Costa,Costa2,Costa3}, both surfaces have Picard number and
geometric Picard number $4$. The discriminant of the Néron-Severi
lattice of the surfaces $Y_{23}$ and $Y_{41}$ have classes modulo
square $-211$ and $-1$ respectively, thus the Picard number of $Y_{/\CC}$
is $3$ and $Y$ is an example of surface of type $S_{2}$.
\end{example}

\subsection{Construction of surfaces of type $S_{3}$\label{subsec:ExamplesS3}}

Let $Y$ be a smooth quartic surface with two hyperplanes sections
$H_{1},H_{2}$ such that each of them is the union of two conics $H_{1}=C_{1}+C_{2}$
and $H_{2}=C_{3}+C_{4}$. Then
\begin{prop}
\label{prop:Construction-S3}The intersection matrix of the curves
$C_{1},\dots,C_{4}$ is either 
\[
M_{1}=\left(\begin{array}{cccc}
-2 & 4 & 1 & 1\\
4 & -2 & 1 & 1\\
1 & 1 & -2 & 4\\
1 & 1 & 4 & -2
\end{array}\right)\text{ or }M_{2}=\left(\begin{array}{cccc}
-2 & 4 & 0 & 2\\
4 & -2 & 2 & 0\\
0 & 2 & -2 & 4\\
2 & 0 & 4 & -2
\end{array}\right).
\]
Suppose that $Y$ has Picard number $3$. If the intersection matrix
is $M_{1}$ then the Néron-Severi lattice of $Y$ is isometric to
$S_{3}$, in particular $Y$ has trivial automorphism group and the
$4$ conics are the only $(-2)$-curves on $Y$. \\
If the intersection matrix is $M_{2}$ then the $4$ conics are also
the only $(-2)$-curves on $Y$, this is case $S_{114}$ in \cite[Table 3]{Nikulin2}. 
\end{prop}

\begin{proof}
Since the curves $C_{1},C_{2}$ and $C_{3},C_{4}$ are plane conics
we have $C_{1}C_{2}=4=C_{3}C_{4}$ and $C_{j}H_{k}=2$. Let $a\in\{0,1,2\}$
be such that $a=C_{1}C_{3}$, then $C_{1}C_{4}=2-a$, $C_{3}C_{2}=2-a$
and $C_{2}C_{4}=a$. Thus, up to exchanging $C_{3}$ and $C_{4}$,
the intersection matrix of the $C_{i}$ is either $M_{1}$ or $M_{2}$.\\
 In the first case, the lattice generated by the $4$ conics in $Y$
is (isometric to) $S_{3}$. Suppose that $Y$ has Picard number $3$.
The discriminant group $A(S_{3})$ of $S_{3}$ is generated by $\frac{1}{12}L,\,\frac{1}{3}(A_{1}-A_{2})$.
There is no non-trivial isotropic element in $A(S_{3})$, thus the
Néron-Severi lattice of the quartic $Y$ is $S_{3}$.\\
In the second case, the rank $3$ lattice generated by the $4$ conics
is denoted by $S_{114}$ in \cite[Table 3]{Nikulin2}. Still by \cite[Table 3]{Nikulin2},
there are two over-lattices $S_{112}$ and $S_{111}$ of the lattice
$S_{114}$. But the surfaces with Néron-Severi lattice $S_{112}$
or $S_{111}$ contain only three $(-2)$-curves, thus the Néron-Severi
lattice of $Y$ is $S_{114}$, and by \cite{Nikulin2} such a surface
contains only four $(-2)$-curves. 
\end{proof}
In the second case the fundamental domain $\mathcal{F}_{X}$ is not
compact, we also remark from the intersection matrix $M_{2}$ that
it has an elliptic fibration; we give more details on that surface
in \cite{Roulleau}. 
\begin{cor}
The moduli space of K3 surfaces $X$ with $\NS X)\simeq S_{3}$ is
unirational.
\end{cor}

\begin{proof}
Let $C_{1}$ and $C_{3}$ be two conics on two fixed planes $P_{1},P_{2}$
in $\PP^{3}$ (we can choose $P_{j}=\{x_{j}=0\},\,j\in\{1,2\}$),
such that $C_{1},C_{3}$ meet transversely at exactly one point. Let
$Y$ be a smooth quartic in the linear system of quartics containing
$C_{1}$ and $C_{3}$. The conics $C_{2},C_{4}$ such that $P_{1}\cdot Y=C_{1}+C_{2}$,
$P_{2}\cdot Y=C_{3}+C_{4}$ are smooth if $Y$ is sufficiently generic.
Then, because $C_{1},C_{3}$ meet transversely at one point, the intersection
matrix of $C_{1},\dots,C_{4}$ is $M_{1}$, thus the Picard number
of $Y$ is at least $3$ and for a generic choice it is $3$, then
by Proposition \ref{prop:Construction-S3}, $\NS Y)\simeq S_{3}$.
 From that construction, it is clear that the moduli space is unirational.
\end{proof}
To show the effectiveness of the construction, let us give an example
of a K3 surface with Néron-Severi lattice $S_{3}$.
\begin{example}
\label{exa:Costa}Consider the quartic $Y$ with equation
\[
\begin{array}{c}
-10x^{3}z+8x^{3}t-180x^{2}t^{2}-9xy^{3}-2xy^{2}z-10xy^{2}t+9xyt^{2}+430xz^{3}-784xz^{2}t\\
-678xzt^{2}+680xt^{3}+3y^{4}+16y^{3}z-43y^{2}z^{2}-1896y^{2}zt-2825y^{2}t^{2}-4246yz^{3}\\
-4700yz^{2}t+23846yzt^{2}+32620yt^{3}-420z^{4}+376z^{3}t+2048z^{2}t^{2}-376zt^{3}-1628t^{4}.
\end{array}
\]
One can check that the hyperplane sections $H_{1}=\{x=0\}$ and $H_{2}=\{y=0\}$
are the union of conics $C_{1},C_{2},C_{3},C_{4}$ with intersection
matrix $M_{1}$.  The reduction $Y_{23}$ and $Y_{73}$ modulo $23$
and $73$ of $Y$ are smooth K3 surfaces. Thanks to the computations
of E. Costa \cite{Costa4}, the Picard number of both surfaces $Y_{23}$
and $Y_{73}$ is $4$ and this is also their geometric Picard number.
The class modulo squares of the discriminant of the Néron-Severi lattice
of $Y_{23}$ is $-2\cdot3\cdot11\cdot17$ whereas for $Y_{73}$, it
is $-6$, thus (see Proposition \ref{prop:The-surface-3conics} and
its proof) the surface $Y$ has Picard number $3$ and is an example
of surface of type $S_{3}$. 
\end{example}

\subsection{Surfaces of type $S_{4}$\label{subsec:ExamplesS4}}

Let $Y$ be a surface of which is the double cover of the plane $\pi:Y\to\PP^{2}$
branched over a sextic curve $C_{6}$ that possess a tritangent line
$L$ and a $6$-tangent conic $C_{o}$. Let $B_{1}+B_{2}=\pi^{*}L$
and $B_{3}+B_{4}=\pi^{*}C_{o}$ be the $(-2)$-curves over the line
and the conic. 
\begin{prop}
There are two possibilities for the intersection matrix of the $(-2)$-curves
$B_{1},\dots,B_{4}$. Either it is 
\[
M_{1}=\left(\begin{array}{cccc}
-2 & 3 & 1 & 1\\
3 & -2 & 1 & 1\\
1 & 1 & -2 & 6\\
1 & 1 & 6 & -2
\end{array}\right)\text{ or it is }M_{2}=\left(\begin{array}{cccc}
-2 & 3 & 0 & 2\\
3 & -2 & 2 & 0\\
0 & 2 & -2 & 6\\
2 & 0 & 6 & -2
\end{array}\right).
\]
Suppose that $Y$ has Picard number $3$ and the intersection matrix
is $M_{1}$. Then $Y$ is a K3 surface of type $S_{4}$. \\
If $Y$ has Picard number $3$ and the intersection matrix is $M_{2}$,
then the $(-2)$-curves $B_{1},\dots,B_{4}$ are the only $(-2)$-curves
on $Y$, this is case $S_{113}$ in \cite[Table 3]{Nikulin2}. 
\end{prop}

In the second case the surface $Y$ has not a bounded fundamental
domain (see \cite{Nikulin2}). In fact, as it can be read of from
the intersection matrix, the surface $Y$ has two elliptic fibrations
(see \cite{Roulleau}). Both moduli of K3 surfaces have the same dimension
i.e. $17$. 
\begin{proof}
It is easy to check that up to permuting $B_{3}$ and $B_{4}$, the
matrices $M_{1}$ and $M_{2}$ are the only possibilities for the
intersection of the curves $B_{j}$. Let us suppose that $Y$ has
Picard number $3$. If the intersection matrix is $M_{1}$ then the
lattice generated by the $(-2)$-curves is $S_{4}$; the discriminant
group of $S_{4}$ (of order $20$) has no non-trivial isotropic element,
thus there is no over-lattice containing $S_{4}$ and $\NS Y)\simeq S_{4}$.\\
 Suppose that the intersection matrix is $M_{2}$. The divisor $D=B_{1}+B_{2}$
is ample with $D^{2}=2$ and the degrees of $B_{1},B_{2},B_{3},B_{4}$
with respect to $D$ are $1,1,2,2$ respectively. The lattice generated
by the $(-2)$-curves is the lattice denoted $S_{113}$ in \cite[Table 3]{Nikulin2}.
There is only one over-lattice of $S_{113}$, but by \cite[Table 3]{Nikulin2},
a surface with that lattice has only $3$ $(-2)$-curves, therefore
$\NS Y)\simeq S_{113}$.
\end{proof}
It is not so easy to find an example of a K3 surface with Néron-Severi
lattice $S_{4}$. We found this one:
\begin{example}
Let $f_{3},q_{2},q_{4}$ be the forms 
\[
\begin{array}{c}
f_{3}=-xy^{2}+2y^{3}+5x^{2}z+3xyz+yz^{2}-6z^{3},\,\,\,\,q_{2}=2x^{2}+xy+2y^{2}-3z^{2},\\
q_{4}=-3x^{4}+11x^{3}y+13x^{2}y^{2}+17xy^{3}+2y^{4}+15x^{2}yz+4xy^{2}z-y^{3}z\hfill\\
+14x^{2}z^{2}-3xyz^{2}-2y^{2}z^{2}+2yz^{3}-12z^{4}.\hfill
\end{array}
\]
The conic $\{q_{2}=0\}$ is $6$-tangent to the sextic $C_{6}=\{q_{2}q_{4}-f_{3}^{2}=0\}$.
Moreover one can check that the line $L=\{y=0\}$ is tritangent to
$C_{6}$ at points $(-1:0:1),\,(0:0:1),\,(1:0:1)$. Let $Y\to\PP^{2}$
be the double cover branched over $C_{6}$. Let $Y_{q}$ be the reduction
in finite field $F_{q}$ of the surface $Y$. Using Magma \cite{magma},
one computes that surfaces $Y_{31^{2}}$ and $Y_{53^{2}}$ have Picard
number $4$, and this is also their geometric Picard number. Moreover,
the class in $\QQ^{\times}/(\QQ^{\times})^{2}$ of the discriminant
of the Néron-Severi lattice of $Y_{31^{2}}$ is $-83$, whereas the
class of the discriminant of $Y_{53^{2}}$ is $-5\cdot53$, therefore,
as in Proposition \ref{prop:The-surface-3conics}, the surface $Y$
has Picard number $3$. One can check moreover that the intersection
of the $(-2)$-curves on $Y$ is (up to ordering) the matrix $M_{1}$,
thus $Y$ has type $S_{4}$. 
\end{example}

\subsection{Construction of surfaces of type $S_{5}$\label{subsec:ExamplesS5}}

Let $Y$ be the double cover of the plane branched over a smooth sextic
curve which has two tritangent lines $L_{1},L_{2}$.
\begin{prop}
\label{prop:S5ok}Suppose that the Picard number of $Y$ is $3$.
Then the Néron-Severi lattice of the K3 surface $Y$ is isometric
to $S_{5}$.
\end{prop}

\begin{proof}
Let $\pi:Y\to\PP^{2}$ be the double cover and let $B_{1},\dots,B_{4}$
be the $(-2)$-curves such that 
\[
\pi^{*}L_{1}=B_{1}+B_{2},\,\pi^{*}L_{2}=B_{3}+B_{4}.
\]
We must have $(B_{1}+B_{2})^{2}=2$ and $(B_{3}+B_{4})^{2}=2$, thus
$B_{1}B_{2}=B_{3}B_{4}=3$. Moreover using the involutions from the
double plane structure, one find relations of type
\[
B_{1}(B_{3}+B_{4})=\frac{1}{2}\pi^{*}L_{1}\pi^{*}L_{2}=1.
\]
Finally, up to permuting $B_{3}$ and $B_{4}$, the intersection matrix
$(B_{i}B_{j})_{i,j}$ is the one given for $S_{5}$. Thus the rank
$3$ Néron-Severi lattice of $Y$ contains $S_{5}$. Since the discriminant
group of $S_{5}$ has no non-trivial isotropic element, we have the
reverse inclusion.
\end{proof}
Let $l_{1},l_{2},f_{3},q_{4}$ be two linear forms, a cubic form
and a quartic and let $f_{6}=l_{1}l_{2}q_{4}-f_{3}^{2}.$ Suppose
that these forms are generic so that the sextic curve $C_{6}=\{f_{6}=0\}$
is smooth. Then the sextic and the line $L_{i}=\{l_{i}=0\}$ are tangent
at the $3$ intersection points of $L_{i}$ and the cubic $E=\{f_{3}=0\}$.
Conversely, we have:
\begin{thm}
\label{prop:PlaneSexticTwoTritangent}Let $C=\{f_{6}=0\}$ be a generic
smooth plane sextic curve which possesses two tritangent lines $L_{1}=\{l_{1}=0\}$,
$L_{2}=\{l_{2}=0\}$. There exists a quartic form $q_{4}$ and a cubic
form $f_{3}$ such that $f_{6}=l_{1}l_{2}q_{4}-f_{3}^{2}.$
\end{thm}

\begin{proof}
Let us study the dimension of the moduli space of smooth sextics $C_{6}$
which admit equations of the form $f_{6}=l_{1}l_{2}q_{4}-f_{3}^{2}$. 

A K3 surface $Y$ which is the double cover of the plane branched
over a generic equation of the form $l_{1}l_{2}q_{4}-f_{3}^{2}=0$
has Néron-Severi lattice of rank $\geq3$. By the same argument as
in the proof of Theorem \ref{thm:S1-equation-of-C6}, Example \ref{exa:example-S5}
below shows that the generic such surface $Y$ has Picard number $3$,
and therefore $\NS Y)\simeq S_{5}$. In particular, the lines $l_{i}=0,\,i\in\{1,2\}$
are determined by $Y$ : if $(l,l',f,q)$ are forms of degree $1,1,3$
and $4$ respectively such that $l_{1}l_{2}q_{4}-f_{3}^{3}=ll'q-f^{3}$,
then there exists scalars $\l,\mu$ such that $l=\l l_{1},l'=\mu l_{2}$.
The map $\Phi$ which to a quadruple $(l_{1},l_{2},q_{4},f_{3})$
associate $l_{1}l_{2}q_{4}-f_{3}^{2}$ is invariant under the action
of the two dimensional family of automorphisms $\G:(l_{1},l_{2},q_{4},f_{3})\to(\l l_{1},\mu l_{2},\frac{1}{\l\mu}q_{4},f_{3})$,
$\l,\mu\in\CC^{*}$. Using such map $\G$, one can suppose $l=l_{1},l'=l_{2}$,
and we have 
\begin{equation}
l_{1}l_{2}q_{4}-f_{3}^{2}=l_{1}l_{2}q-f^{2},\label{eq:egalite2f6}
\end{equation}
if and only if 
\[
l_{1}l_{2}(q_{4}-q)=(f_{3}+f)(f_{3}-f).
\]
Suppose there exists some quadric forms $c_{1},c_{2}$ such that 
\[
l_{1}c_{1}=(f_{3}+f),\,\,l_{2}c_{2}=(f_{3}-f),
\]
then $f_{3}=\frac{1}{2}(l_{1}c_{1}+l_{2}c_{2})$, but the curve $C_{6}$
would be singular at the intersection point of $L_{1}$ and $L_{2}$,
thus that case cannot happen. Up to exchanging $f$ with $-f$ and
after a short computation, it turns out that equation \eqref{eq:egalite2f6}
is equivalent to the existence of a linear form $l$ such that 
\[
\left\{ \begin{array}{c}
f=\pm(f_{3}+l_{1}l_{2}l)\\
q=(q_{4}+2lf_{3}+l_{1}l_{2}l^{2}).
\end{array}\right.
\]
The map $\Phi$ is a morphism from a $31=3+3+15+10$ dimensional vector
space with $5$ dimensional fibers (because of the choice of $l\in H^{0}(\PP^{2},\OO(2))$
and the above parameters $\l,\mu$), thus the image of $\Phi$ is
a $26$ dimensional variety in $H^{0}(\PP^{2},\OO_{\PP^{2}}(6))$.
The moduli of smooth sextics curves which admits equations of the
form $f_{6}=l_{1}l_{2}q_{4}-f_{3}^{2}$ is the quotient of that variety
by $GL_{3}(\CC)$, thus it is $17$ dimensional. By Proposition \ref{prop:S5ok},
the double cover of the plane branched over the generic curve $C_{6}$
is a K3 surface of type $S_{5}$. The moduli space of K3 surfaces
of type $S_{5}$ is irreducible and $17$ dimensional. This moduli
is also the moduli of sextic curves which possess two tritangent lines,
thus the result.
\end{proof}
From the proof of Theorem \ref{prop:PlaneSexticTwoTritangent} one
gets a dominant rational map from a $31$ dimensional vector space
to the moduli space of sextic with two tritangent lines, and therefore
to the moduli of K3 surface of type $S_{5}$. Thus we obtain
\begin{cor}
The moduli space of K3 surfaces of type $S_{5}$ is unirational.
\end{cor}

Let $C_{6}=\{l_{1}l_{2}q_{4}-f_{3}^{2}=0\}$ be a plane sextic curve
with two tritangent lines and let $\pi:X\to\PP^{2}$ be the double
cover branched over $C_{6}$. 
\begin{rem*}
The sextic curve $C_{6}$ is tangent to the quartic curve $Q_{4}=\{q_{4}=0\}$
at the $12$ intersection points of $Q_{4}$ with $E=\{f_{3}=0\}$.
The pull back $\pi^{*}Q_{4}$ splits: $\pi^{*}Q_{4}=C_{1}+C_{2}$
with $C_{1}^{2}=C_{2}^{2}=4$ and $C_{1}C_{2}=12$. If the Picard
number of $X$ is 3, one can check that the numerical equivalence
classes of $C_{1}$ and $C_{2}$ are the two big and nef divisors
of square $4$ on $X$ (see Section \ref{subsec:N=0000E9ron-Severi-S5}),
and thus the linear systems $|C_{1}|$ and $|C_{2}|$ gives the two
morphisms $X\to\PP^{3}$. From the proof of Proposition \ref{prop:PlaneSexticTwoTritangent},
the pull-back of every quartic of the form $q_{4}+2lf_{3}+l_{1}l_{2}l^{2}=0$
(where $l$ is a linear form) also splits; that gives the linear systems
$|C_{1}|$ and $|C_{2}|$.
\end{rem*}
In order to complete the proof of Theorem \ref{prop:PlaneSexticTwoTritangent},
let us give the following Example:
\begin{example}
\label{exa:example-S5}By taking 
\[
\begin{array}{c}
l_{1}=5x+z,\,\,l_{2}=2z,\,\,f_{3}=13x^{2}y+7xy^{2}+4y^{3}+14x^{2}z+xyz+20y^{2}z+8xz^{2},\\
q_{4}=9x^{4}+4x^{3}y+2x^{2}y^{2}+6xy^{3}+2y^{4}+12x^{3}z+6x^{2}yz+11xy^{2}z\hfill\\
+7y^{3}z+17x^{2}z^{2}+5xyz^{2}+16y^{2}z^{2}+16xz^{3}+3yz^{3}+8z^{4},
\end{array}
\]
we obtain a smooth sextic curve $C_{6}=\{l_{1}l_{2}q_{4}-f_{3}^{2}=0\}$.
Let $Y$ be the double cover of the plane branched over $C_{6}$.
The $\cu$-curves above the lines $l_{i}=0$ generate a lattice isometric
to $S_{5}$. The reduction of $Y$ modulo $29$ and $53$ have Picard
number $4$, and this is also the geometric Picard number. By using
the same technics as in Proposition \ref{prop:The-surface-3conics},
we get in the first case that the class of the discriminant in $\QQ^{\times}/(\QQ^{\times})^{2}$
is $11\cdot41$ and in the second case that the class of the discriminant
is $59$, thus the surface $Y$ cannot have Picard number equal to
$4$ and it is a surface of type $S_{5}$. 
\end{example}

\subsection{Construction of surfaces of type $S_{6}$}

Let $Y\to\PP^{2}$ be the double cover of the plane branched over
a generic sextic $C_{6}$ which possesses a $6$-tangent conic $Q_{1}$
and two cuspidal cubics $G_{1},G_{2}$ such that the irreducible components
of the pull-back $\pi^{*}G_{1}=B_{3}+B_{4}$, $\pi^{*}G_{2}=B_{5}+B_{6}$
are smooth $(-2)$-curves. 
\begin{prop}
The surface $Y$ has type $S_{6}$.
\end{prop}

\begin{proof}
Let $B_{1},B_{2}$ be the $(-2)$-curves above the conic $Q_{1}$.
The intersection matrix of the curves $B_{1},\dots,B_{6}$ is 
\[
\left(\begin{array}{cccccc}
-2 & 6 & u & \bar{u} & \bar{v} & v\\
6 & -2 & \bar{u} & u & v & \bar{v}\\
u & \bar{u} & -2 & 11 & \tilde{a} & a\\
\bar{u} & u & 11 & -2 & a & \tilde{a}\\
\bar{v} & v & \tilde{a} & a & -2 & 11\\
v & \bar{v} & a & \tilde{a} & 11 & -2
\end{array}\right)
\]
where $\bar{x}=6-x$ and $\tilde{a}=9-a$, and moreover $u,v\in\{0,...,6\}$,
$a\in\{0,...,9\}$. Up to exchanging $B_{3}$ with $B_{4}$ and $B_{5}$
with $B_{6}$, we can suppose that $u\leq3$ and $v\leq3$. One can
then check that this matrix has rank $3$ if and only if $u=v=1$
and $a=9$. This is then the intersection matrix of the Néron-Severi
lattice of a K3 surface of type $S_{6}$. 
\end{proof}
Let us take the notations of sub-section \ref{subsec:Surfaces-with-N=0000E9ron-SeveriS6}.
The divisor $D_{6}=A_{1}+A_{5}$ is very ample, of square $6$, with
\[
D_{6}A_{j}=3,7,1,14,3,12,\text{ for }j=1,\dots,6.
\]
The linear system $|D_{6}|$ defines a degree $6$ model in $\PP^{4}$
for which the curves $A_{1},A_{5}$ are rational normal cubic curves
contained in an hyperplane section, and $A_{3}$ is a line, which
meets $A_{1}$ but not $A_{5}$. Let us prove
\begin{prop}
The moduli space of K3 surfaces with $\NS X)\simeq S_{6}$ is unirational.
\end{prop}

\begin{proof}
Let $H\hookrightarrow\PP^{4}$ be a hyperplane and let $C_{1}\hookrightarrow H$
be a rational normal quartic curve. Let $C_{3}$ be a line not contained
in $H$ and that cuts (transversely) $C_{1}$ at one point. The linear
system of quadrics (respectively cubics) containing $C_{1}$ and $C_{3}$
has dimension $5$ (respectively $21$). \\
Let $Y$ be the complete intersection K3 surface of a generic quadric
and cubic in these linear systems. The hyperplane section of $Y$
by $H$ decomposes as $C_{1}+C_{5},$ where $C_{5}$ is a degree $3$
curve such that $C_{5}^{2}=-2$. Since $(C_{1}+C_{5})C_{2}=1$, we
have $C_{2}C_{5}=0$. The intersection matrix of $(C_{1},C_{3},C_{5})$
is the same as the curves $(A_{1},A_{3},A_{5})$ in sub-section \ref{subsec:Surfaces-with-N=0000E9ron-SeveriS6},
thus the lattice $L_{1}$ in $\NS Y)$ generated by $C_{1},C_{2},C_{3}$
is isometric to $S_{6}$. The embedding $L_{1}\hookrightarrow\NS Y)$
is primitive since $L_{1}$ has no over-lattice. The family of K3
surfaces $Y$ we constructed is unirational. Moreover it contains
the family of surfaces with $\NS Y)\simeq S_{6}$ as an open subset,
since surfaces with $\rho>3$ and containing strictly $L_{1}$ form
an hypersurface in the moduli space of K3 with a polarization $L_{1}\hookrightarrow\NS Y)$.
We therefore conclude that the moduli of K3 surfaces with $\NS X)\simeq S_{6}$
is unirational.
\end{proof}

\subsection{Surfaces with Picard number $4$ of type $L(24)$}

Let $C_{6}\hookrightarrow\PP^{2}$ be a smooth plane sextic curve
which admits $3$ tritangent lines $L_{j}=\{l_{j}=0\}$ $j\in\{1,2,3\},$
$l_{j}\in H^{0}(\PP^{2},\OO(1))$. Let $\pi^{*}:Y\to\PP^{2}$ be the
double cover branched over $C_{6}$. The pull-back of the lines splits
as $\pi^{*}L_{j}=B_{2j-1}+B_{2j}$ $(j\in\{1,2,3\}$) where $B_{1},\dots,B_{6}$
are $(-2)$-curves. 
\begin{prop}
Suppose that the Picard number of $Y$ is $4$. Up to exchanging $B_{3}$
with $B_{4}$ and $B_{5}$ with $B_{6}$, the intersection matrix
of the curves $B_{1},\dots,B_{6}$ is one of the two following matrices:
\[
M_{1}=\left(\begin{array}{cccccc}
-2 & 3 & 0 & 1 & 0 & 1\\
3 & -2 & 1 & 0 & 1 & 0\\
0 & 1 & -2 & 3 & 0 & 1\\
1 & 0 & 3 & -2 & 1 & 0\\
0 & 1 & 0 & 1 & -2 & 3\\
1 & 0 & 1 & 0 & 3 & -2
\end{array}\right),\,M_{2}=\left(\begin{array}{cccccc}
-2 & 3 & 0 & 1 & 0 & 1\\
3 & -2 & 1 & 0 & 1 & 0\\
0 & 1 & -2 & 3 & 1 & 0\\
1 & 0 & 3 & -2 & 0 & 1\\
0 & 1 & 1 & 0 & -2 & 3\\
1 & 0 & 0 & 1 & 3 & -2
\end{array}\right).
\]
In the first case the K3 surface $Y$ has type $L(24)$ studied in
Section \ref{subsec:Lattice-L24}. \\
In the second case the $(-2)$-curves $B_{1},\dots,B_{6}$ are also
the only $(-2)$-curves on $Y$, which is a surface with finite automorphism
group.
\end{prop}

\begin{proof}
One has $B_{1}B_{2}=B_{3}B_{4}=B_{5}B_{6}=3$ and by the same computations
as in Section \ref{subsec:ExamplesS5}, the other intersection number
$B_{i}B_{j}$ are in $\{-2,0,1\}$ and up to exchanging $B_{1},B_{2}$
and $B_{3},B_{4}$, one can suppose $B_{1}B_{3}=0,$ which forces
$B_{1}B_{4}=0$, $B_{2}B_{3}=1$ and $B_{2}B_{4}=0$. Also up to exchanging
$B_{5}$ with $B_{6}$, we can suppose that their intersection with
$B_{1},B_{2}$ is as given by the two matrices. It remains two possibilities
for the intersections of $B_{3},B_{4}$ with $B_{5}$ and $B_{6}$,
which give matrices $M_{1}$ and $M_{2}$. The case of matrix $M_{1}$
(the corresponding lattice has discriminant $-28$) is (up to permutation
of the curves $B_{i}$) the intersection matrix given in Section \ref{subsec:Lattice-L24}.
\\
If the intersection matrix is $M_{2}$, the elements $B_{1},B_{2},B_{3},B_{5}$
form a base of the lattice $\L$ generated by the $B_{i}$. The lattice
$\L$ (of discriminant $-27$) is the lattice named $L(25)$ in \cite{Vinberg},
which verifies: 
\[
\L\simeq\left[\begin{array}{cc}
0 & 3\\
3 & -2
\end{array}\right]\oplus\mathbf{A}_{2}(-1).
\]
It has discriminant group $\ZZ/9\ZZ\times\ZZ/3\ZZ$ which has no non-trivial
isotropic element, thus the Néron-Severi lattice of $Y$ is $\L$.
Such a K3 surface also contains only a finite number of $(-2)$-curves
(see \cite{Vinberg}).
\end{proof}
\begin{rem}
If $Y$ is generic and its Néron-Severi lattice is $M_{2}$, one can
prove that the sextic curve possesses an equation of the form $l_{1}l_{2}l_{3}g-f^{2}=0,$
where the $l_{i}$ are the equations of the tritangent lines and $g,f$
are cubic forms (more details are given in \cite{Roulleau}). The
surface $Y$ has two elliptic fibrations coming from the pull-back
of the cubic curve $\{g=0\}$ which is $9$-tangent to the sextic
curve. 
\end{rem}

It is not so easy to find an example of surface of type $L(24)$,
here is an example:
\begin{example}
Let $l_{1},l_{2},q_{4},f$ be the forms 
\[
\begin{array}{c}
l_{1}=x+y+2z,\,\,\,\,l_{2}=-3x+2y+z,\hfill\\
q_{4}:=8x^{4}+x^{3}y+x^{2}y^{2}+3xy^{3}-2y^{4}-20x^{3}z-2x^{2}yz-xy^{2}z+3y^{3}z\\
-12x^{2}z^{2}+xyz^{2}+4yz^{3},\\
f=5x^{3}-3x^{2}y+xy^{2}+4y^{3}+2x^{2}z-3xyz-3y^{2}z+5xz^{2}+4yz^{2}.\hfill
\end{array}
\]
The smooth sextic curve $C_{6}=\{l_{1}l_{2}q_{4}-f^{2}\}$ has three
tritangent lines which are $l_{1}=0$, $l_{2}=0$ and the line $\{y=0\}$.
Let $Y\to\PP^{2}$ be the double cover branched over $C_{6}$. Using
the same technics as in Proposition \ref{prop:The-surface-3conics},
one computes that the Picard number and the geometric Picard number
of its reduction $Y_{13}$ modulo $13$ is $4$. Moreover, by the
Artin-Tate formula, the discriminant of the Néron-Severi lattice of
$Y_{13}$ is $-28$, thus $Y$ is a surface of type $L(24)$ (alternatively,
we computed the intersection matrix of the $6$ $(-2)$-curves one
the lines and checked that they generate the rank $4$ lattice $L(24)$).
\end{example}

\subsection{Surfaces with Picard number $4$ of type $L(27)$}

Let $C_{6}$ be a plane sextic curve which has one tritangent line
$L_{1}=\{l_{1}=0\}$ and three $6$-tangent conics. Let $\pi:Y\to\PP^{2}$
be the double cover branched over $C_{6}$. 
\begin{prop}
The surface $Y$ is either a K3 surface of type $L(27)$ or a K3 surface
such that the Néron-Severi lattice has a base with intersection matrix
\[
M_{2}=\left(\begin{array}{cccc}
-2 & 0 & 0 & 0\\
0 & -2 & 4 & 0\\
0 & 4 & -2 & 1\\
0 & 0 & 1 & -2
\end{array}\right).
\]
Both cases exists; in the second case $Y$ has infinitely many $(-2)$-curves
and automorphisms. 
\end{prop}

\begin{proof}
The pull-back of the conics splits: $\pi^{*}C_{j}=B_{2j-1}+B_{2j}$
$(j\in\{1,2,3\}$) with $B_{1}B_{2}=B_{3}B_{4}=B_{5}B_{6}=6$. Let
$B_{7},B_{8}$ be the $(-2)$-curves above the line $L_{1}$. As in
Sections \ref{subsec:ExamplesS1} and \ref{subsec:ExamplesS4}, there
exist $a,b,c\in\{0,1,...,4\}$ and $u,v,w\in\{0,1,2\}$ such that
\[
(B_{i}B_{j})_{1\leq i,j\leq8}=\left(\begin{array}{cccccccc}
-2 & 6 & a & \bar{a} & b & \bar{b} & u & \tilde{u}\\
6 & -2 & \bar{a} & a & \bar{b} & b & \tilde{u} & u\\
a & \bar{a} & -2 & 6 & c & \bar{c} & v & \tilde{v}\\
\bar{a} & a & 6 & -2 & \bar{c} & c & \tilde{v} & v\\
b & \bar{b} & c & \bar{c} & -2 & 6 & w & \tilde{w}\\
\bar{b} & b & \bar{c} & c & 6 & -2 & \tilde{w} & w\\
u & \tilde{u} & v & \tilde{v} & w & \tilde{w} & -2 & 3\\
\tilde{u} & u & \tilde{v} & \tilde{v} & \tilde{w} & w & 3 & -2
\end{array}\right),
\]
where $\bar{x}=4-x$ and $\tilde{w}=2-x$. Up to exchanging $B_{3}$
with $B_{4}$, $B_{5}$ with $B_{6}$ and $B_{7}$ with $B_{8}$,
we can suppose $a,b\in\{0,1,2\}$ and $u\in\{0,1\}$. The matrix $(B_{i}B_{j})$
has rank $4$ if and only if $(a,b,c,u,v,w)$ is one of the following
$6$-tuples 
\[
\begin{array}{cc}
T_{1}=(0,0,4,1,1,1), & T_{6}=(2,0,4,0,2,2),\\
T_{2}=(0,0,4,0,1,0), & T_{7}=(0,2,4,0,2,2),\\
T_{3}=(0,0,4,1,2,0), & T_{8}=(0,0,2,0,2,2),\\
T_{4}=(0,0,4,1,0,2), & T_{9}=(0,2,0,0,2,0),\\
T_{5}=(0,0,4,0,0,1), & T_{10}=(2,0,0,0,0,2)
\end{array}
\]
The first case is the lattice $L(27)$. The lattices corresponding
to cases $T_{2},\dots,T_{5}$ are isometric and lead to a K3 surface
with a Néron-Severi lattice of rank $4$ and intersection matrix $M_{2}$
in some base. For that base, the class $D=(-1,1,1,0)$ is ample and
satisfies $D^{2}=2$; the classes $(0,0,0,1)$, $(-1,1,1,-1)$ are
the classes of irreducible $(-2)$-curves of degree $1$ for $D$.
The classes 
\[
\begin{array}{c}
(1,0,0,0),\,\,(0,1,0,0),\,\,(0,0,1,0),\\
(-3,2,2,0),\,(-2,1,2,0),\,(-2,2,1,0)
\end{array}
\]
are the classes of irreducible $(-2)$-curves of degree $2$ for $D$.
Suppose that $|D|$ has some base points. Then there exists an elliptic
curve $E$ and $(-2)$-curve $\G$ such that $D=2E+\G$ and $\G E=1$.
But that would imply $D\G=0$, which is impossible since $D$ is ample,
thus $|D|$ is base point free and it defines a double cover of $\PP^{2}$
branched over a smooth sextic curve which has a tritangent line and
three $6$-tangent conics.\\
The cases $T_{5},\dots,T_{10}$ lead to the same rank $4$ lattice
with intersection matrix
\[
\left(\begin{array}{cccc}
-2 & 0 & 0 & 0\\
0 & -2 & 2 & 2\\
0 & 2 & -2 & 0\\
0 & 2 & 0 & -2
\end{array}\right),
\]
in some base. This is the Néron-Severi lattice of a K3 surface which
contains only six $(-2)$-curves (case $2U\oplus[-2]\oplus[-2]$ in
\cite{Vinberg}), but this cannot happen in our situation, since the
surface $Y$ contains at least $8$ $(-2)$-curves.
\end{proof}
Let us give one example of surface of type $L(24)$:
\begin{example}
Let $q_{1},q_{2},q_{3}$ and $f$ be the forms
\[
\begin{array}{c}
q_{1}=2x^{2}+xy+2y^{2}-3z^{2},\,\,q_{2}:=3x^{2}+8xy+10y^{2}+11yz\\
q_{3}=-25/3x^{2}-xy-3y^{2}-2xz-yz+5z^{2},\hfill\\
f=2x^{3}-3x^{2}y-3x^{2}z-xyz+y^{2}z-3xz^{2}-3yz^{2}.\hfill
\end{array}
\]
The conics $\{q_{i}=0\},\,i\in\{1,2,3\}$ are $6$-tangent to the
smooth sextic curve $C_{6}=\{q_{1}q_{2}q_{2}-f^{2}=0\}$, moreover
the line $L=\{y=0\}$ is tritangent to $C_{6}$ at points $(-1:0:1),\,(0:0:1),\,(1:0:1)$.
Let $Y\to\PP^{2}$ be the double cover branched over $C_{6}$. One
can check that the reduction $Y_{19}$ modulo $19$ of $Y$ has Picard
number $4$, which is also the geometric Picard number. Using the
Artin-Tate conjecture and Magma \cite{magma} as in Proposition \ref{prop:The-surface-3conics},
one obtains that the discriminant of the Néron-Severi lattice of $Y_{19}$
is $-60$ (and the Brauer group is trivial), thus the surface $Y$
is an example of surface of type $L(27)$ (note that the discriminant
of the lattice $L(27)$ is $-60$ and the discriminant of the lattice
associated to the matrix $M_{2}$ equals $-52$). 

%\newpage
\end{example}

\section{Appendix}

For the convenience of the reader, we give a list of Gram matrices
of the Néron-Severi lattices of the surfaces we consider, with their
number of $(-2)$-curves, their automorphism group and their discriminant
group:

\begin{tabular}{|c|c|c|c|c|}
\hline 
 & $\NS X)$ & $\#$ of $(-2)$-curves & $\aut(X)$ & $\Disc X)$\tabularnewline
\hline 
$S_{1}$ & $\left(\begin{array}{ccc}
6 & 0 & 0\\
0 & -2 & 0\\
0 & 0 & -2
\end{array}\right)$ & $6$ & $\ZZ/2\ZZ$ & $(\ZZ/2\ZZ)^{2}\times\ZZ/6\ZZ$\tabularnewline
\hline 
$S_{2}$ & $\left(\begin{array}{ccc}
36 & 0 & 0\\
0 & -2 & 1\\
0 & 1 & -2
\end{array}\right)$ & $6$ & $1$ & $\ZZ/3\ZZ\times\ZZ/36\ZZ$\tabularnewline
\hline 
$S_{3}$ & $\left(\begin{array}{ccc}
12 & 0 & 0\\
0 & -2 & 1\\
0 & 1 & -2
\end{array}\right)$ & $4$ & $1$ & $\ZZ/3\ZZ\times\ZZ/12\ZZ$\tabularnewline
\hline 
$S_{4}$ & $\left(\begin{array}{ccc}
2 & 1 & 2\\
1 & -2 & 1\\
2 & 1 & -2
\end{array}\right)$ & $4$ & $\ZZ/2\ZZ$ & $\ZZ/20\ZZ$\tabularnewline
\hline 
$S_{5}$ & $\left(\begin{array}{ccc}
4 & 0 & 0\\
0 & -2 & 1\\
0 & 1 & -2
\end{array}\right)$ & $4$ & $\ZZ/2\ZZ$ & $\ZZ/12\ZZ$\tabularnewline
\hline 
$S_{6}$ & $\left(\begin{array}{ccc}
2 & 2 & 3\\
2 & -2 & 1\\
3 & 1 & -2
\end{array}\right)$ & $6$ & $\ZZ/2\ZZ$ & $\ZZ/44\ZZ$\tabularnewline
\hline 
$L_{24}$ & $\left(\begin{array}{cccc}
2 & 1 & 1 & 1\\
1 & -2 & 0 & 0\\
1 & 0 & -2 & 0\\
1 & 0 & 0 & -2
\end{array}\right)$ & $6$ & $\ZZ/2\ZZ$ & $\ZZ/2\ZZ\times\ZZ/14\ZZ$\tabularnewline
\hline 
$L_{27}$ & $\left(\begin{array}{cccc}
12 & 2 & 0 & 0\\
2 & -2 & 1 & 0\\
0 & 1 & -2 & 1\\
0 & 0 & 1 & -2
\end{array}\right)$ & $8$ & $\ZZ/2\ZZ$ & $\ZZ/2\ZZ\times\ZZ/30\ZZ$\tabularnewline
\hline 
\end{tabular}

\vspace{5mm}
\noindent Xavier Roulleau,
\\Aix-Marseille Universit\'e, CNRS, Centrale Marseille,
\\I2M UMR 7373,  
\\13453 Marseille,
\\France
\\ \email{Xavier.Roulleau@univ-amu.fr}
\vspace{0.1cm}

{\small{}http://www.i2m.univ-amu.fr/perso/xavier.roulleau/Site\_Pro/Bienvenue.html}{\small\par}
\end{document}